\documentclass{amsart}
\usepackage{graphicx}
\usepackage{amssymb}
\usepackage{array}
\usepackage{enumerate}
\usepackage{url}
\usepackage{algorithm}
\usepackage{dsfont}
\usepackage{diagbox}
\usepackage{mathtools}
\usepackage{float}
\usepackage{wrapfig}
\usepackage{lscape}
\usepackage{rotating}
\usepackage{epstopdf}

\newcommand{\abs}[1]{\lvert#1\rvert}

\newcommand{\M}{\mathcal{M}}

\newcommand{\floor}[1]{\left\lfloor#1\right\rfloor}
\newcommand{\ceil}[1]{\left\lceil#1\right\rceil}
\newcommand{\frc}[1]{\left\{#1\right\}}

\newcommand{\ud}{\mathrm{d}}

\renewcommand{\v}{\mathbf{v}}

\newcommand{\I}{\mathcal{I}}
\newcommand{\J}{\mathcal{J}}
\newcommand{\jbulk}{\J_1}
\newcommand{\jboundary}{\J_2}

\DeclareMathOperator{\fl}{fl}

\DeclareMathOperator{\re}{Re}
\DeclareMathOperator{\im}{Im}

\newtheorem{theorem}{Theorem}
\newtheorem{proposition}[theorem]{Proposition}
\newtheorem{lemma}[theorem]{Lemma}

\numberwithin{theorem}{section}
\theoremstyle{remark}

\makeatletter
\let\@@pmod\pmod
\DeclareRobustCommand{\pmod}{\@ifstar\@pmods\@@pmod}
\def\@pmods#1{\mkern 2mu({\operator@font mod}\mkern 3mu#1)}
\makeatother

\begin{document}

\title[New Computations of $\zeta(1/2+it)$]{New Computations of the Riemann zeta
function on the critical line}
\author[J.W. Bober \and G.A. Hiary]{Jonathan W. Bober \and Ghaith A. Hiary}
\thanks{GH is partially supported by
the National Science Foundation under agreements No.
 DMS-1406190. Both authors are pleased to thank 
 IAS, MSRI, and ICERM, where parts of this work was conducted}
\address{
    JB: Heilbronn Institute for Mathematical Research, School of Mathematics,
    University of Bristol, Howard House, Queens Avenue, Bristol BS8 1SN,
    United Kingdom
}
\email{j.bober@bris.ac.uk}
\address{
    GH: Department of Mathematics, The Ohio State University, 231 West 18th
    Ave, Columbus, OH 43210, USA
}
\email{hiary.1@osu.edu}
\subjclass[2010]{Primary: 11Y35. Secondary: 65Y20.}
\keywords{The Riemann zeta function, large values, exponential sums, the van der Corput iteration, Theta algorithm}

\begin{abstract}
We present highlights of computations of the Riemann zeta function around large
values and high zeros. The main new ingredient in these computations is an
implementation of the second author's fast algorithm for numerically evaluating
quadratic exponential sums.  In addition, we use a new simple multi-evaluation
method to compute the zeta function in a very small range at little more than
the cost of evaluation at a single point.
\end{abstract}

\maketitle

\section{Introduction}

Computations of $\zeta(1/2+it)$ have a long history 
and are of
interest in number theory because of 
fundamental links to the prime numbers. 
(The primes and zeta zeros are
Fourier transforms of each other.) Progress in zeta computations
has paralleled advances in our understanding of zeta, and 
has in many ways traced the development of the computer.

With the exception of Riemann's hand calculation 
in the 1850s, 
which remained unknown to the
outside world for many decades, 
computations of $\zeta(1/2+it)$ 
until 1932 had relied on 
the Euler-Maclaurin formula for approximating sums by integrals. 
This formula enables the numerical evaluation of $\zeta(1/2+it)$ 
with high accuracy, but requires summing at least $t/2\pi$ terms.
Since computations were done manually, 
this limited the feasible range of $t$ to a few hundred.

An important advance came in 1932 when 
Siegel published on a formula that he discovered in Riemann's Nachclass.
This formula, 
now known as the Riemann--Siegel (RS) formula, allows computing 
$\zeta(1/2+it)$ by summing $\lfloor \sqrt{t/2\pi}\rfloor$ terms, far fewer 
than in the Euler-Maclaurin method.
The resulting accuracy is only moderate (i.e.\ to within $\pm t^{-c}$ for various
$c>0$ and large $t$) but still suffices in most numerical investigations. 
Shortly after Siegel's paper, Titchmarsh~\cite{titchmarsh36} reported in 1936 on 
his computation using the RS formula to
 verify the Riemann hypothesis (RH) up to $t=1468$, extending
the previous verification range by severalfold.\footnote{Puzzlingly,
Titchmarsh
stated in \cite{titchmarsh36}
that zeta has $1041$ zeros between $t=0$ and $t=1468$ even though 
there are $1042$ zeros in that range.}

For a long time after that, progress in zeta computations came from 
advances in computing technology rather than faster algorithms. 
This is nicely illustrated
in Turing's effort in the 1930s to build a mechanical computer to calculate
zeta, and
his pioneering use of the electronic computer for that purpose
 two decades later. 
In 1953, Turing~\cite{turing53}
published on his verification of the RH
in the interval $2\pi 63^2< t< 2\pi 64^2$, 
in which the main sum of the RS formula consists of $63$ terms. 
This interval lies significantly higher on the
critical line than in previous computations. 
As remarked in \cite{hejhal-odlyzko},
Turing's consideration of an isolated interval at high $t$
signaled a great insight.
Certainly, Turing was motivated by his skepticism of the RH as, to quote him, 
his ``calculations were done in an optimistic hope that a zero would be
found off the critical line.''
To check the RH, Turing introduced a novel method 
 to prove that a given list of zeta zeros in an interval is complete.
This method is still the state of the art today. See
\textsection{\ref{turing's method}}.

The main algorithmic improvements on the RS formula 
did not arrive until the 1980s, starting with the  
the Odlyzko--Sch\"onhage algorithm~\cite{odlyzko-schonhage-algorithm} 
for multiple evaluations of zeta, and
 the algorithms of Sch\"onhage~\cite{schonhage} and Heath-Brown (see
\cite{hiary}) for evaluation at a single point.
Odlyzko 
implemented the Odlyzko--Sch\"onhage algorithm
and computed large sets of zeros at various heights
in order to test conjectured statistical 
connections between zeta and random matrix theory.
The most extensive of Odlyzko's computations began in 1999 
and finished in  Fall of 2000,
resulting in a dataset of 20 billion
zeros around zero number $10^{23}$.
Gourdon~\cite{gourdon} then computed two billion zeros around the $10^{24}$-th
zero using another implementation of the Odlyzko--Sch\"onhage algorithm. 

In the meantime, the theoretical complexity of computing $\zeta(1/2+it)$ at a
single point became 
of interest in its own right. 
New algorithms by the second author~\cite{hiary,hiary-2} 
have bounded this complexity by 
$t^{1/3+o(1)}$ time and little space, 
and then by $t^{4/13+o(1)}$ time and space. The
purpose of this article is to report on our implementation of the $t^{1/3+o(1)}$ 
algorithm and subsequent computations. In fact, 
a basic implementation of this algorithm had already been completed 
by the authors during March--July of 2010 while resident at
IAS. This was followed by debugging and refinements 
in short periods of availability after that.

Most of our computations were in
small  ``targeted intervals'' where zeta was expected to be
unusually large. 
Our hope  was that
unusual behaviors of the zeros would be discovered in such intervals. 
The heuristic method used to find
these intervals naturally 
produced very large $t$, and so 
most computations
at large $t$ were out of necessity.
This resulted in 
the largest computed value of $Z(t)\approx 16244.8652$ and $S(t)\approx
3.3455$, and in checking the RH for more than $50000$ zeros 
in over $200$ small intervals going up  
to the $10^{36}$-th zero. See \textsection{\ref{examples of computations}}.

Our computations demonstrate that the $t^{1/3+o(1)}$
algorithm is practical. Indeed, almost 
all computations finished in a few days on the ``riemann machine,'' 
a computer cluster at the University of Waterloo.
(The  implementation source code is available on the authors' websites.)
This cluster has 
16 nodes (though we usually limited our use to 
about 12 nodes) and each node has 8 cores at 2.27GHz clock speed.
During initial development of our code, we also used William Stein's
"Sage cluster" at the University of Washington, and in more recent computations
we used the BlueCrystal cluster at the University of Bristol.
Probably, the range of $t$ feasible via the current implementation and 
using easily available computational resources is past the $10^{40}$-th zero.

\section{The Riemann--Siegel formula}\label{rs formula sec}

Let $s=\sigma+it$ where $\sigma$ and $t$ are real numbers. 
The Riemann zeta function $\zeta(s)$ is defined by
\begin{equation}
\zeta(s) :=  \sum_{n=1}^{\infty} n^{-s} \qquad (\sigma > 1).
\end{equation}
It has a meromorphic continuation with a simple pole at $s=1$, 
and satisfies
the functional equation $\xi(s)=\xi(1-s)$ 
where $\xi(s) := \pi^{-s/2}\Gamma(s/2) \zeta(s)$.
The RH is the conjecture that all 
zeros (i.e.\ roots) of $\xi(s)$ lie on the critical line $\re(s)=1/2$.

The RS formula is an asymptotic formula to compute $\zeta(s)$.
The starting point 
is the following identity, proved via the functional equation and general principles 
in complex analysis. Let $\chi(s):=\pi^{s-1/2} \Gamma((1-s)/2)/\Gamma(s/2)$, 
and let
\begin{equation}
\mathcal{R}(s):=\frac{e^{-\pi i s}\Gamma(1-s)}{2\pi i}\int_{C_N}
\frac{z^{s-1}e^{-Nz}
}{e^z-1}dz,
\end{equation}
where $C_N$ is the contour that goes from $+\infty$ to $(2N+1)\pi i$, makes half a circle
of radius $(2N+1)\pi$ around the origin, then returns to $+\infty$. 
(Here, $z^{s-1} = e^{(s-1)\log z}$ where the logarithm is 
real at the beginning of
$C_N$.) Then we have
\begin{equation}
\zeta(s) = \sum_{n=1}^N \frac{1}{n^s}+\chi(s)\sum_{n=1}^N \frac{1}{n^{1-s}}
+  \mathcal{R}(s).
\end{equation}
Riemann analyzed the remainder $\mathcal{R}(s)$ using a saddle-point method,  
and obtained a remarkably simple evaluation as an asymptotic series.\footnote{
Edwards' book~\cite[page 156]{edwards} contains
a copy of a page of Riemann's notes where
he does such calculations. As a curiosity,
we compared with C.L.~Siegel's 1932 paper
``\"Uber Riemanns Nachla\ss zur analytischen Zahlentheorie,'' and it
appears that the 
magnified expression at the bottom of that copy corresponds
to the fourth correction term $C_4(z)$; see \eqref{correction terms sum}.}

Actually, the RS formula is usually stated for
$Z(t):=e^{i\theta(t)}\zeta(1/2+it)$,
which is 
 a rotated version of zeta on the critical line often called the Hardy $Z$-function.
 The phase factor $\theta(t)$ is
a real smooth function, determined via the functional equation, 
that ensures that $Z(t)$
is real for real $t$. 
In particular, 
one can isolate
simple non-trivial zeros of zeta
merely by looking for sign changes of $Z(t)$. 
It is thus desirable to work with 
$Z(t)$ instead of $\zeta(1/2+it)$.
There is no harm in doing so because these functions can be 
recovered easily from each
other. For by Lemma~\ref{theta phase lemma} we have 
\begin{equation}\label{thetat}
\left|\theta(t) - \left(\frac{1}{2}\log\frac{t}{2\pi e}
-\frac{\pi}{8} + \frac{1}{48t}\right)\right| \le \frac{0.129}{t^3}, \qquad (t > 1). 
\end{equation}
So one can compute $\theta(t)$ precisely and quickly for large $t$.
Also, since $Z(t)=Z(-t)$ by the functional equation, 
we may restrict to $t\ge 0$.

The version of the RS formula 
that we used in our computations is the same as in
\cite{rubinstein-computational-methods}. For $t>2\pi$,
let $a:=\sqrt{t/(2\pi)}$, $N:=\lfloor a\rfloor$ the
integer part of $a$, and $z= a-\lfloor a\rfloor$ the 
fractional part of $a$.
Then the RS formula consists of a main sum 
\begin{equation}
\mathcal{M}(t) = \sum_{n=1}^{N} 
\frac{e^{it \log n}}{\sqrt{n}},
\end{equation}
a correction $\mathcal{C}_m(t)$, and a remainder $R_m(t)$.  
Specifically, for each $m\in \mathbb{Z}_{\ge 0}$,
\begin{equation} \label{eq:rsform}
Z(t) = 2\, \Re e^{-i\theta(t)}\mathcal{M}(t) 
+\mathcal{C}_m(t) + R_m(t),
\end{equation}
where $\mathcal{C}_m(t)$ is a sum of $m+1$ terms,
\begin{equation}\label{correction terms sum}
\mathcal{C}_m(t)=\frac{(-1)^{N+1}}{\sqrt{a}}
\sum_{r=0}^m \frac{C_r(z)}{a^r},
\end{equation}
and $R_m(t)$ satisfies the bound $R_m(t)\ll t^{-(2m+3)/4}$. 
The functions $C_r(u)$ on the r.h.s.\ in \eqref{correction terms sum}
are linear combinations of derivatives
of
\begin{equation}\label{F(u) def}
F(u):=\frac{\cos (2\pi(u^2-u-1/16))}{\cos(2\pi u)},
\end{equation}
up to the $3r$-th derivative; see
\cite{gabcke-thesis,berry-riemann-siegel,berry-keating-zeta-method}.\footnote{Slightly different 
expressions for the functions $C_r(z)$ can be found in
\cite{odlyzko-schonhage-algorithm, berry-riemann-siegel} 
and in Siegel's 1932 paper.}
For example, 
$C_0(u) = F(u)$ and $C_1(u) = F^{(3)}(u)/96\pi^2$,
where $F^{(3)}(u)$ is the third derivative of $F(u)$.
The arguments of the cosines in \eqref{F(u) def} are simultaneously
 odd multiples of $\pi/2$, so that poles due to denominator 
 are cancelled by zeros due to the numerator.
Therefore, the function $F(u)$ is entire and even (but not periodic). 
To evaluate $F(u)$, and more generally $C_r(u)$, near removable singularities at 
$u=\pi(2n+1)/2$, we used a numerically stable Taylor expansion.

The remainder $R_m(t)$ in \eqref{eq:rsform} 
is well-controlled when $t$ is large.
Gabcke derived in his thesis~\cite{gabcke-thesis} 
explicit bounds for $R_m(t)$ when $m\le 10$.
For example, if $t\ge 200$, then
$|R_1(t)| < .053 t^{-5/4}$, $|R_4(t)| <  0.017t^{-11/4}$,
and $|R_{10}(t)| < 25966 t^{-23/4}$.
Actually, even the bound for $R_1(t)$ 
will be sufficient in our computations since 
 $t$ will be of size $ > 10^{20}$, 
 so already $|R_1(t)|< 10^{-26}$.

\section{Turing's method}\label{turing's method}

The zeros of $\xi(s)$ are called the non-trivial zeros of zeta. 
They are denoted by $\rho_n=1/2+i\gamma_n$, $n\ne 0$. For example, 
$\rho_1 = 1/2+i 14.134725\ldots$, $\rho_2=1/2+ i 21.022039\ldots$,
 $\rho_3 = 1/2+ i 25.010857\ldots$, and $\rho_{-n}=\overline{\rho_n}$. 
The RH is the statement that the 
ordinates $\gamma_n$ are always real.
The counting function of zeros is 
\begin{equation}
N(t) := |\{0< \im \rho_n < t\}|+\frac{1}{2}|\{\im \rho_n = t\}|.
\end{equation}
Using the functional equation and 
the argument principle from complex analysis,
if $t\ne \im \rho_n$ for any $n$, then
\begin{equation}\label{Nt formula}
N(t) = \frac{1}{\pi} \theta(t)+1 + S(t),
\end{equation}
where $\theta(t)$ and $S(t)$ are defined\footnote{This is the same
$\theta(t)$ appearing in the definition of $Z(t)$.} by a continuous
variation in the argument of 
$\pi^{-s/2}\Gamma(s/2)$ and 
$\zeta(s)$, respectively, as $s$ moves 
along the line segments from $2$, where the argument is defined to be zero, 
to $2+i t$ to $1/2+it$;
see \cite{titchmarsh,davenport-book,edwards}. 
If $t=\im \rho_n$ for some $n$, then we define 
$S(t):= \lim_{\epsilon \searrow 0} \frac{1}{2}(S(t+\epsilon)+S(t-\epsilon))$. 

As mentioned earlier, $\theta(t)$ can be computed easily; it increases
roughly linearly with no unpredictable oscillations. 
Therefore, in view of formula \eqref{Nt formula}, 
the main difficulty is to compute $S(t)$,
for which we employ Turing's method~\cite{turing53}.
This is a particularly attractive method as  
 it uses the  already computed list 
of  zeros together with a few evaluations of $Z(t)$. 
Basically, the value of $S(t)$ is determined at two points $t_2>t_1$,
which in turn determines $N(t_2)-N(t_1)$.\footnote{We
assume that $Z(t_1)$ and $Z(t_2)$ are nonzero.}
If the number of zeros found in $[t_1,t_2]$ 
matches this difference, 
then the completeness of the zeros list is verified,
and we are in a position to 
compute $S(t)$ throughout $[t_1,t_2]$.
It is thus 
clear that the main issue is to find
such $t_1$ and $t_2$.

To this end, Turing first observed that 
if the sign of $Z(t)$ is known, then the value of $S(t)$ is known modulo
$2$. It is therefore not necessary to compute $S(t)$ to any great accuracy.
In fact, as observed in \cite{edwards}, it is advantageous to specialize to
$t$ a good gram point\footnote{The $m$-th gram point $g_m$ is 
the unique solution the equation $\theta(t)= \pi m \in 
\mathbb{Z}_{\ge -1}$ for $t\ge 7$. It is called good 
if $(-1)^m Z(g_m)>0$. One usually finds a good gram point on 
testing $Z(t)$ at few consecutive gram points.}, 
for then $S(t)$ must be an even integer, 
and it suffices to prove that $|S(t)|<2$ in order to conclude that $S(t)=0$.
The basic idea here is that $S(t)$ is small on average, satisfying 
the bound $|\int_t^{t+\Delta} S(y)dy| \le 0.128 \log(t+\Delta)+ 2.30$,
provided that 
$t>168\pi$ and $\Delta >0$. (The constants in this bound have been improved
by Trudgian~\cite{trudgian}.) So if one incorrectly assumes that $S(t)\ge
2$, then, provided that sign changes of the $Z$-function 
are sufficiently regularly spaced around $t$, 
 the average of $S(y)$ over $[t,t+\Delta]$ will contradict 
the required bound once $\Delta$ is large enough (roughly of size $\gg
\log(t+\Delta)$). And therefore one can conclude that $S(t)<2$. 
An analogous argument can be used to prove 
that $S(t) > -2$; see \cite{turing53,edwards} for details.

\section{Examples of computations}\label{examples of computations}

The new methods described in this paper are suitable for evaluation of
$Z(t)$ in short intervals. Accordingly, most of our computations
have focused on evaluating the zeta function high on the critical line at
spots where we might expect interesting behavior. Additionally, we have done
some evaluation at spots where $t$ or $N(t)$ is a nice ``round'' number; in
these spots we expect to see ``typical'' behavior. 

To find points where we expect to see interesting behavior, 
 we used the LLL algorithm~\cite{lll}, as done in \cite{odlyzko-manuscript}, 
 to search for values of $t$ where $p^{it} \approx 1$
for many initial primes $p$.
By multiplicativity, then, there should be unusually many values of
$n^{it}$ close to $1$, making the initial segment of the main sum large. 
Though
we have not attempted to make this argument rigorous, it works well in practice,
and we have observed many values of $Z(t)$ which are much larger than average
by using LLL to line up the values of just one hundred or so initial $p^{it}$.

As a byproduct of our search for large values, we also find large values of
$S(t)$. 
It is always the case in our computations that when $\zeta(1/2 + it)$
is very large there is a large gap between the zeros around the large
value. 
And it seems that to compensate for
this large gap the zeros nearby get ``pushed'' to the left and right. A typical
trend in the large values that we have found is that $S(t)$ is particularly
large and positive before the large value and large and negative afterwards.
This behavior can be seen in the plots in Figures \ref{fig-large-S} and
\ref{fig-large-Z}.

As a consequence on the Riemann Hypothesis, $\zeta(1/2 + it)$ is known to grow
slower than any fixed power of $t$. Currently, the best conditional upper bound is
\begin{equation}\label{chandee-sound}
    \abs{\zeta(1/2 + it)} \ll \exp\left(\frac{\log 2}{2} \frac{\log t}{\log\log t} +
        O\left(\frac{\log t \log\log\log t}{(\log\log t)^2}\right)\right),
\end{equation}
a result of Chandee and Soundararajan \cite{chandee-sound} improving the leading constant
in earlier results. On the other hand,
Bondarenko and Seip \cite{bondarenko-seip} have recently shown unconditionally
that there exist
values of $t$ for which
\[
    \abs{\zeta(1/2 + it)} > \exp\left( \left(\frac{1}{\sqrt 2} + o(1)\right)
            \sqrt{\frac{\log t \log\log\log t}{\log\log t}}\right),
\]
improving on previous results by a factor of $\sqrt{\log\log\log t}$. Where
exactly in this range the largest values lie is still an open question, though
a detailed conjecture of Farmer, Gonek, and Hughes \cite{farmer-gonek-hughes}
suggests that the largest values are of size
\begin{equation}\label{fgh}
    \exp\left(\big(1 + o(1)\big)\sqrt{\frac{1}{2}\log t \log\log t}\right).
\end{equation}

Table \ref{table-large-Z}
has a list of the 12 local maxima of $|Z(t)|$ larger than $10000$ 
that we have found to date.
We do not attempt further 
here to shed light on conjectures regarding the
maximal size of $|\zeta(1/2 + it)|$, but focus on presenting some raw data. 
Additionally,
it seems unlikely that the values that we have found are close to as
large as possible. These extreme values are very rare, and cannot be found by
random search, so we find them using one specialized method, 
but there could be other methods.

Bounds for the growth of $S(t)$ are similar to the bounds for the growth of
$\log\abs{\zeta(1/2 + it)}$. Again assuming the Riemann Hypothesis,
Montgomery established in \cite{montgomery} 
that there are values of $t$ for which
\begin{equation}
    |S(t)| \gg \left(\frac{\log t}{\log\log t}\right)^{1/2}.
\end{equation}
Goldston and Gonek proved that $|S(t)| \le (1/2+o(1))\log t/\log\log t$, 
and the constant was improved to $1/4$ in 
\cite{carneiro-chandee-milinovich}.
Even unconditionally, $|S(t)|$ is known to be unbounded 
by results of Selberg. Nevertheless, previous to
these computations, the largest observed value of $S(t)$ seems to have been
$-2.9076$, as reported by Gourdon~\cite{gourdon}. Table \ref{table-large-S} lists
11 spots where we have found values of $|S(t)| > 3.1$, the largest of which
is $S(t) \approx 3.3455$ for $t \approx 7.75 \times 10^{27}$.

In addition to computation at points where we expect $\zeta(1/2+it)$ to be large,
we have computed $Z(t)$ for some arbitrarily chosen values of $t$. Often
these values have either $t$ or $N(t)$ close to some nice round number
in base ten, and the computations highest in the critical strip are for
$$t\approx 81029194732694548890047854481676712.98790,$$ 
the imaginary part of zero number 
$n=10^{36} + 42420637374017961984$.

The data is just meant to be a sample of what we have computed. 
Further examples of computations can be found
on the authors' websites.
All together,
we have checked the Riemann Hypothesis for more than 50000 zeros in over 200
separate small intervals with values of $t$ ranging from $10^{24}$ to $8 \times
10^{34}$. 

\renewcommand\arraystretch{1.2}
\begin{table}
\small
\begin{tabular}{|r|r|} \hline \multicolumn{1}{|c|}{$t$} & \multicolumn{1}{c|}{$Z(t)$} \\ \hline
 $39246764589894309155251169284104.0506$ & $16244.8652$ \\
 $70391066310491324308791969554453.2490$ & $-14055.8928$ \\
 $552166410009931288886808632346.5052$ & $-13558.8331$ \\
 $35575860004214706249227248805977.2412$ & $13338.6875$ \\
 $6632378187823588974002457910706.5963$ & $12021.0940$ \\
 $698156288971519916135942940460.3337$ & $11196.7919$ \\
 $289286076719325307718380549050.2563$ & $10916.1145$ \\
 $50054757231073962115880454671617.4008$ & $-10622.1763$ \\
 $803625728592344363123814218778.1993$ & $10282.6496$ \\
 $690422639823936254540302269442.4854$ & $10268.7134$ \\
 $1907915287180786223131860607197.5463$ & $10251.5994$ \\
 $9832284408046499500622869540131.7445$ & $-10138.5908$ \\ \hline
\end{tabular}
\vspace{.1in}
\caption{All local maxima of $|Z(t)| > 10000$ found by our computations}
\label{table-large-Z}
\end{table}

\begin{table}
\small
\begin{tabular}{|r|r|} \hline \multicolumn{1}{|c|}{$t$} & \multicolumn{1}{c|}{$S(t)$} \\ \hline
 $7757304990367861417150213053.6386$ & $3.3455$ \\
 $546577562321057124801498516819.4609$ & $-3.2748$ \\
 $35575860004214706249227248805976.9763$ & $3.2722$ \\
 $31774695316763918183637654364.8066$ & $3.2573$ \\
 $11580026442432493576924087062.5414$ & $-3.2371$ \\
 $10758662450340950434456735185.3359$ & $-3.2261$ \\
 $50054757231073962115880454671617.8419$ & $-3.1826$ \\
 $50054757231073962115880454671617.8419$ & $-3.1826$ \\
 $39246764589894309155251169284103.7774$ & $3.1694$ \\
 $10251393160473423776137882271.3031$ & $3.1660$ \\
 $77590565202125505656738011641.6876$ & $3.1431$ \\ \hline
\end{tabular}
\vspace{.1in}
\caption{Spots where $|S(t)| > 3.1$ found by our computations. In this table, $t$ actually
denotes the imaginary part of a zero of $\zeta(s)$, so the value of $S(t)$ is attained
just before (if negative) or after this zero.}
\label{table-large-S}
\end{table}

\begin{table}
\small
    \begin{tabular}{|l|r|} \hline
        \multicolumn{1}{|c|}{$n$} & \multicolumn{1}{c|}{$\gamma_n$} \\ \hline
        $10^{25}$ & $1194479330178301585147871.32909$ \\
        $10^{26}$ & $11452628915113964213507127.18757$ \\
        $10^{27}$ & $109990955615748542241920621.36163$ \\
        $98094362213058141112271181439$ & $10^{28} + 0.00366$ \\
        $10^{29}$ & $10191135223869807023206505980.23860$ \\
        $10^{30} + 484$ & $98297762869274424758690514889.09764$ \\
        $1017590402074552798166351185765$ & $10^{29} + 0.07316$ \\
        $10^{31}$ & $949298829754554964058786559878.40484$ \\
        $10^{32}$ & $9178358656494989336431259004805.28194$ \\
        $10^{33}$ & $88837796029624663862630219091104.93992$ \\
        $10^{36} + 42420637374017961984$ & $81029194732694548890047854481676712.98790$ \\
        \hline
    \end{tabular}
    \vspace{.1in}
    \caption{
        Some examples of zeros of $\zeta(1/2 + it)$. Here the notation means
        that $\zeta(1/2 + i\gamma_n) = 0$ for some real number $\gamma_n$
        within about $10^{-5}$ of the number in the right column, and that
        there are exactly ${n-1}$ zeros in of $\zeta(s)$ in the critical
        strip with positive imaginary part $< \gamma_n$.}
    \label{table-round-points}
\end{table}

\clearpage
\begin{sidewaysfigure}
\vspace{120mm}
    \makebox[\textwidth][c]{\includegraphics[scale=.5]{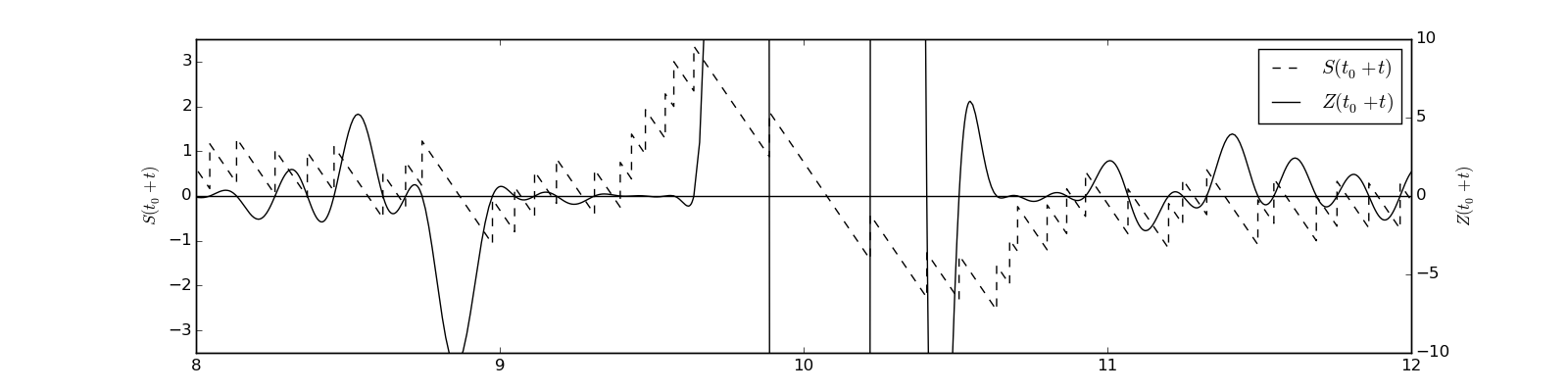}}
    \makebox[\textwidth][c]{\includegraphics[scale=.5]{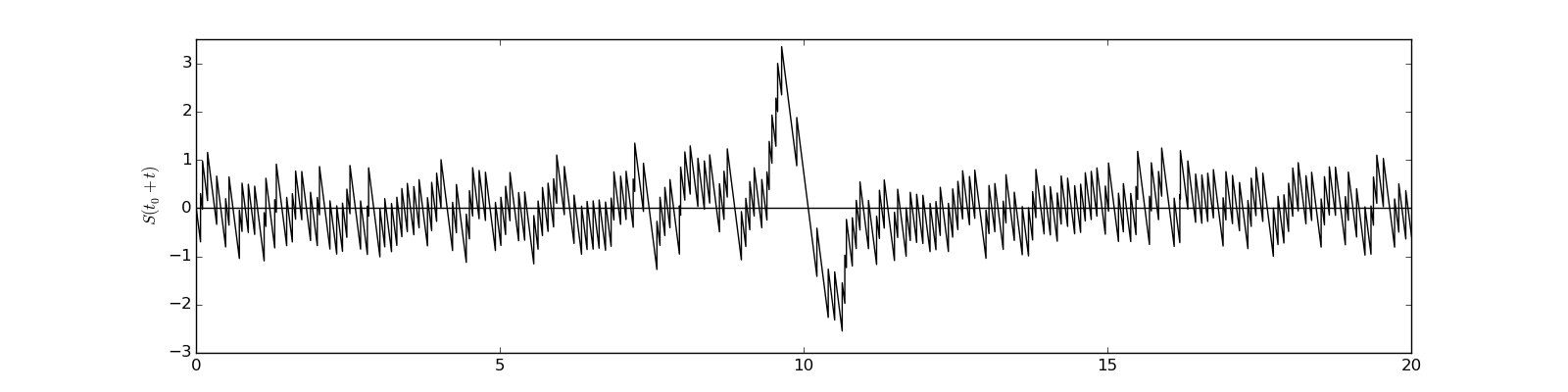}}
    \makebox[\textwidth][c]{\includegraphics[scale=.5]{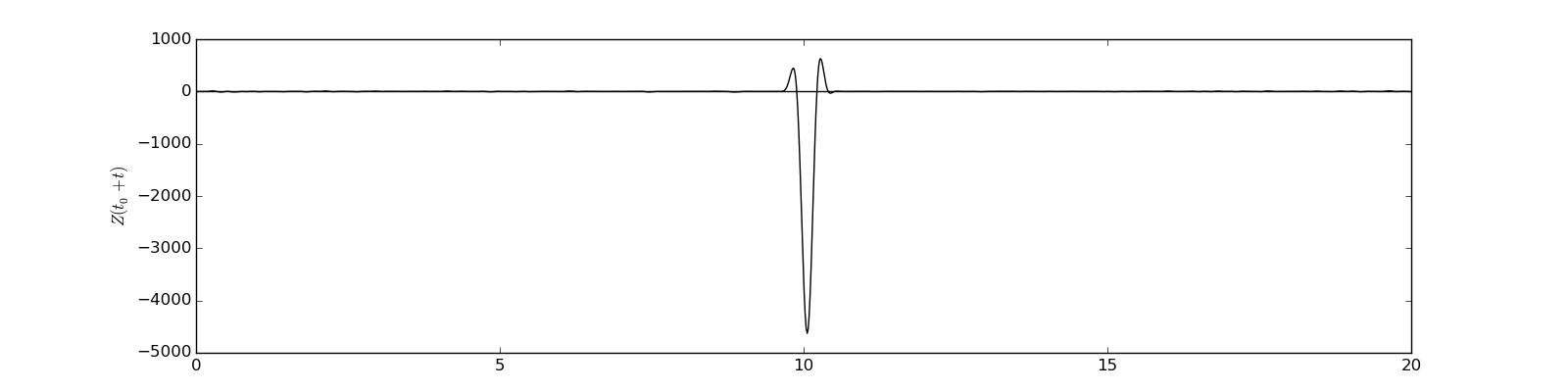}}
    \caption{$Z(t)$ and $S(t)$ around the largest value of $S(t)$ we have found. Here $t_0 = 7757304990367861417150213044$.}\label{fig-large-S}
\end{sidewaysfigure}

\clearpage
\begin{sidewaysfigure}
\vspace{120mm}
    \makebox[\textwidth][c]{\includegraphics[scale=.5]{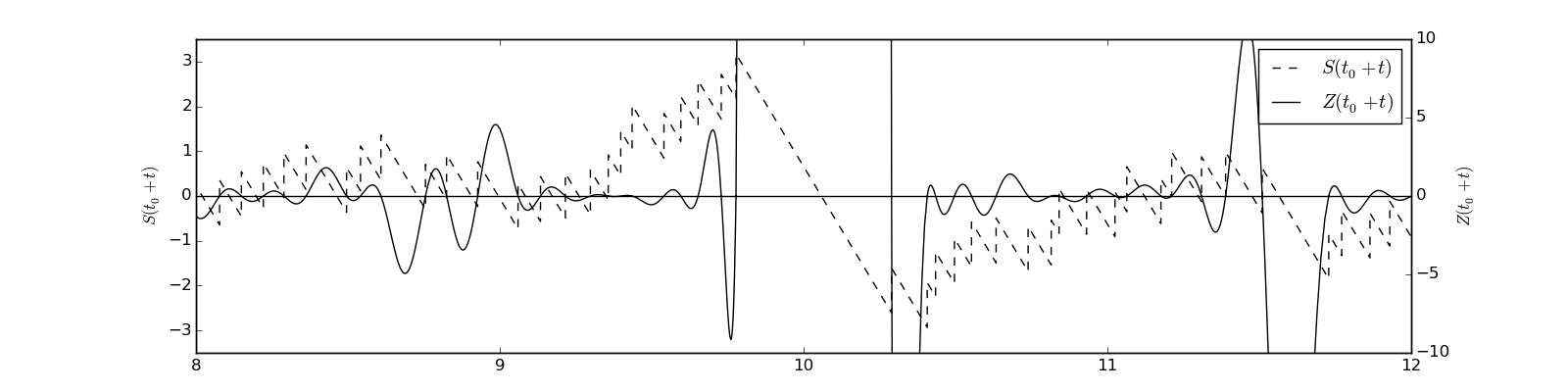}}
    \makebox[\textwidth][c]{\includegraphics[scale=.5]{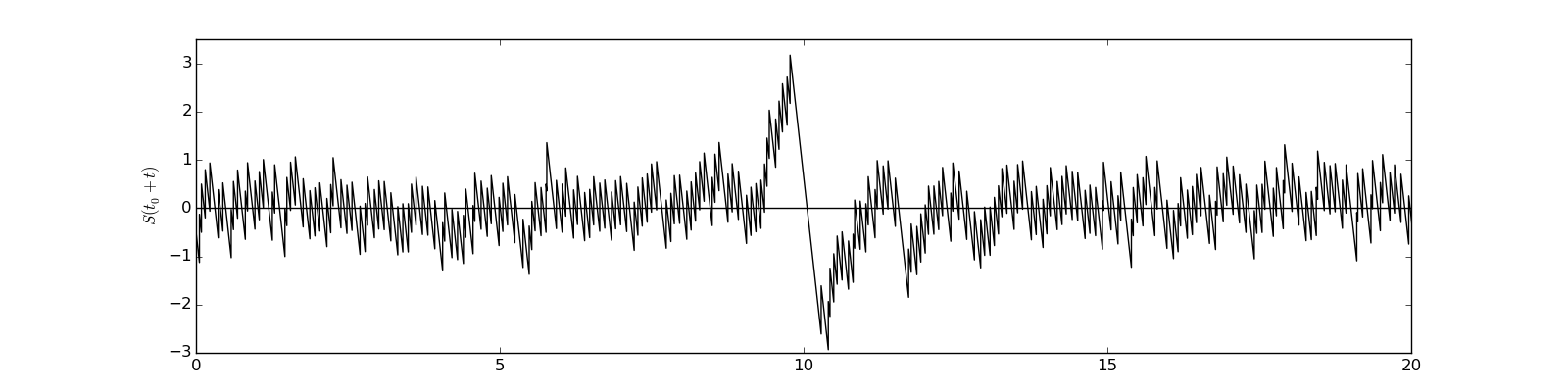}}
    \makebox[\textwidth][c]{\includegraphics[scale=.5]{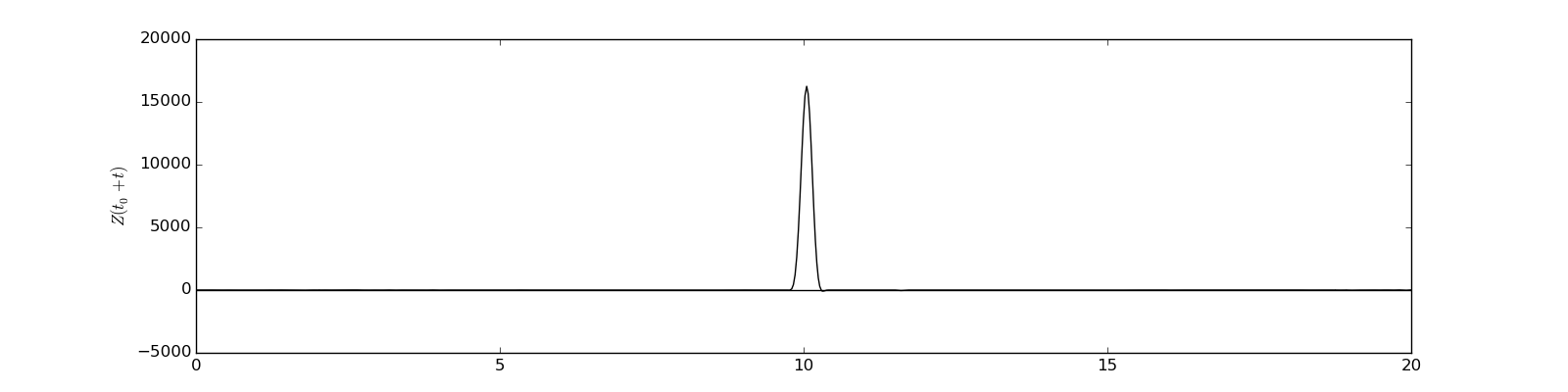}}
    \caption{$Z(t)$ and $S(t)$ around the largest value of $Z(t)$ we have found.
             Here $t_0 = 39246764589894309155251169284094$.}\label{fig-large-Z}
\end{sidewaysfigure}

\clearpage
\begin{sidewaysfigure}
\vspace{120mm}
    \makebox[\textwidth][c]{\includegraphics[scale=.5]{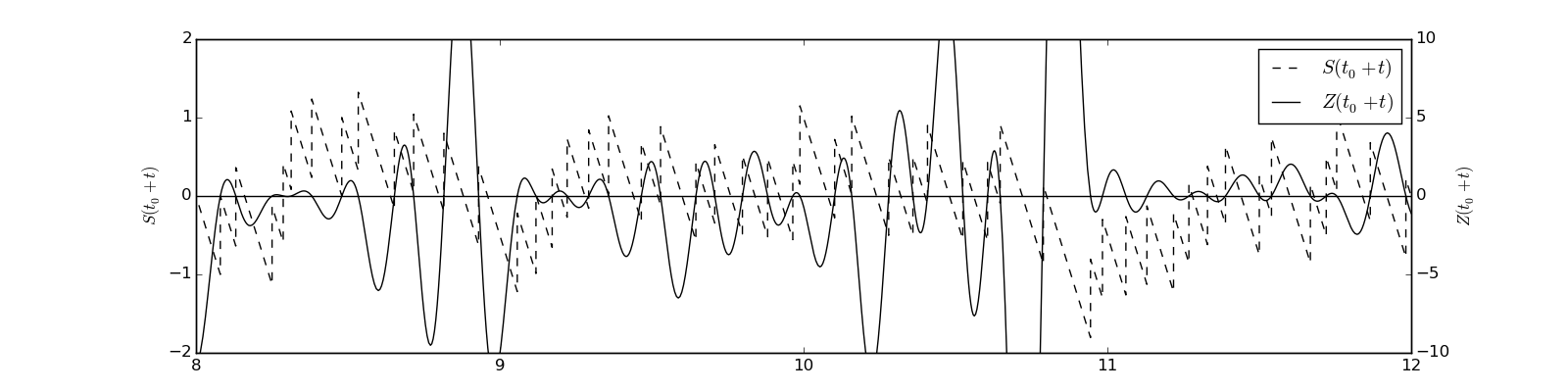}}
    \makebox[\textwidth][c]{\includegraphics[scale=.5]{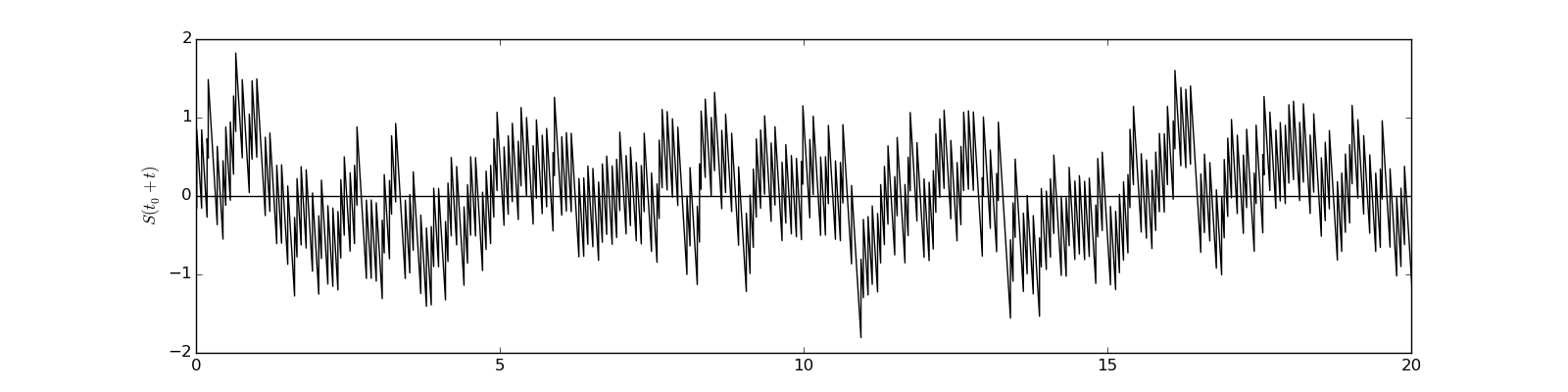}}
    \makebox[\textwidth][c]{\includegraphics[scale=.5]{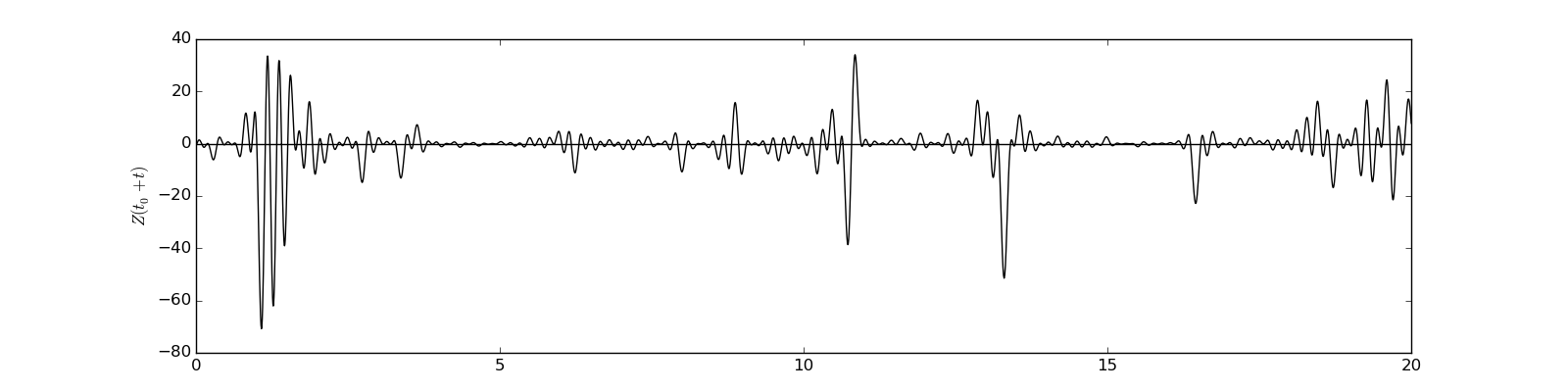}}
    \caption{$Z(t)$ and $S(t)$ around the zero number $1000000000000000042420637374017961984 > 10^{36}$,
        where $t_0 = 81029194732694548890047854481676703$.
    In contrast to Figures \ref{fig-large-S} and \ref{fig-large-Z}, there is nothing particularly special about
this range of $t$, so we should expect this to exhibit typical behavior of $S(t)$ and $Z(t)$.}
    \label{fig-large-t}
\end{sidewaysfigure}

\clearpage
\section{Reducing to quadratic exponential sums} \label{zeta theta link}

To start, we choose a positive integer $v_0 \le N+1$ 
and a real number $u_0>1$. 
Then we construct the sequences 
\begin{equation}\label{seq def}
K_r= \lceil v_r/u_0\rceil,\qquad 
v_{r+1}= v_r + K_r, \qquad (0\le r< R), 
\end{equation}
where
$R$ is the largest integer such that $v_R < N+1$.
We define 
\begin{equation}
K_R :=\min\{\lceil v_R/u_0\rceil,N+1-v_R\}.
\end{equation}
So $v_{R+1} = N+1$, which
is the first point outside
 the range of summation of the main sum $\mathcal{M}(t)$. 
Given the sequences \eqref{seq def}, we subdivide $\mathcal{M}(t)$
into an initial sum of $v_0-1$ terms, followed by
 $R+1$ consecutive blocks where
the $r$-th block starts at $v_r$ and consists of $K_r$ consecutive
terms.
We use the Taylor expansion 
to express the $r$-th block 
as a linear combination of the quadratic exponential sums
\begin{equation}
F(K,j;a,b)=\frac{1}{K^j}\sum_{k=0}^{K-1} k^j e^{2\pi i a k+2\pi i b k^2},\qquad j\in \mathbb{Z}_{\ge 0}.
\end{equation}
(We treat $k^j$ as $1$ when $k = j = 0$.) To this end, we define 
\begin{equation}
f_s(z) := \frac{e^{(s-1/2)(z- z^2/2)}}{(1+z)^s},\qquad  f_s(0)=1.
\end{equation}
For $n\in [v_r,v_r+K_r)$, we write $n=v_r+k$ where $0\le k<K_r$. 
Then, letting $s^* = 1/2-it$ and noting that 
\begin{equation}
\frac{e^{it\log n}}{\sqrt{n}}  =  
\frac{e^{it\log v_r}}{\sqrt{v_r}} e^{itk/v_r-itk^2/2v_r^2} f_{s^*}(k/v_r),
\end{equation}
we obtain, on expanding $f_{s^*}(z)$ around $0$, that 
\begin{equation}\label{taylor exp}
\sum_{v_r\le n<v_r+K_r} \frac{e^{it\log n}}{\sqrt{n}}
=\frac{e^{it\log v_r}}{\sqrt{v_r}}
\sum_{j=0}^{\infty} 
c_r(j) F(K_r,j;a_r,b_r),
\end{equation}
where the linear and  quadratic arguments $a_r$ and $b_r$
are given by the formulas
\begin{equation}\label{coeff def}
a_r := \frac{t}{2\pi v_r},\qquad b_r :=-\frac{t}{4\pi v_r^2}.
\end{equation} 
The coefficients $c_r(j)$ are given by 
\begin{equation}\label{cj def}
c_r(j) := \left(\frac{K_r}{v_r}\right)^j\frac{f_{s^*}^{(j)}(0)}{j!},
\end{equation}
where $f_s^{(j)}(z)$ is the $j$-th derivative in $z$. 
For instance, 
\begin{equation}
\begin{split}
\frac{f_s^{(1)}(z)}{1!} = -\frac{1}{2},\qquad \frac{f_s^{(2)}(z)}{2!}= \frac{3}{8},
\qquad  \frac{f_s^{(3)}(z)}{3!}= -\frac{1}{48}(16s+7).
\end{split}
\end{equation}
Note that, in order to avoid clutter, 
we suppressed dependence on $t$, $u_0$, and $v_0$
in all of $c_r(j)$, $a_r$, $b_r$, and $R$.

Let $J$ denote a positive integer, 
and let $\mathbf{c}_r=(c_r(0),\ldots,c_r(J))$ denote 
a tuple of the first $J+1$ coefficients. Define
\begin{equation}
\mathbf{F}(K,\mathbf{c};a,b)= \sum_{j=0}^J c(j) F(K,j;a,b),
\end{equation}
which is a linear combination of quadratic sums.
We bound the truncation error $\epsilon_J(t)$ 
in the main sum
when the series on the r.h.s.\ of \eqref{taylor exp} is stopped after
$J+1$ terms for each block.
That is, $\epsilon_J(t)$ satisfies 
\begin{equation}\label{quadratic sum formula}
\mathcal{M}(t) = \sum_{n=1}^{v_0-1} \frac{e^{it\log n}}{\sqrt{n}} 
+ \sum_{r=0}^R \frac{e^{it\log v_r}}{\sqrt{v_r}} 
\mathbf{F}(K_r,\mathbf{c}_r;a_r,b_r)+ \epsilon_J(t),
\end{equation}
The following proposition gives bounds on the size of $\epsilon_J(t)$
for various $J$ and $t$.
The proof of this proposition is in \textsection{\ref{proofs}}. 

\begin{proposition}\label{trunc error proposition}
Let $\alpha = 1/0.9$, $K_{\min} = 2000$.
If $v_0 = \lceil K_{\min} t^{1/3}\rceil \le \frac{1}{3}\sqrt{t/(2\pi)}$ say, 
and $u_0= \alpha |s|^{1/3}\ge 1$, then
$\epsilon_J(t)$ is bounded as in the following table.
\begin{table}[h]
\small
\label{eps J table}
\begin{tabular}{|l|l| l| l| l| l|}
\hline
\backslashbox[1cm]{$t$}{$J$} &
 \multicolumn{1}{|c|}{$18$} &
 \multicolumn{1}{c|}{$21$} &
 \multicolumn{1}{c|}{$24$} &
 \multicolumn{1}{c|}{$27$} &
 \multicolumn{1}{c|}{$30$} \\
\hline
$10^{24}$ & $0.000831893$ & $0.0000386099$ & $0.0000108755$ & $7.46138\times 10^{-6}$ & $5.42799\times 10^{-6}$\\
$10^{26}$ & $0.00283154$ & $0.0000925556$ & $7.50489\times 10^{-6}$ & $3.85022\times 10^{-6}$ & $2.76739\times 10^{-6}$\\
$10^{28}$ & $0.00839067$ & $0.000256977$ & $9.10602\times 10^{-6}$ & $1.82102\times 10^{-6}$ & $1.21035\times 10^{-6}$\\
$10^{30}$ & $0.0230403$ & $0.000698646$ & $0.0000196934$ & $1.1228\times 10^{-6}$ & $4.96568\times 10^{-7}$\\
$10^{32}$ & $0.0603446$ & $0.00182714$ & $0.0000495358$ & $1.44903\times 10^{-6}$ & $2.13244\times 10^{-7}$\\
$10^{34}$ & $0.15309$ & $0.00463399$ & $0.000124900$ & $3.12046\times 10^{-6}$ & $1.36063\times 10^{-7}$\\
$10^{36}$ & $0.37952$ & $0.011489$ & $0.00030938$ & $7.53404\times 10^{-6}$ & $1.89985\times 10^{-7}$\\
\hline
\end{tabular}
\vspace{3mm}
\caption{Bounds on $\epsilon_J(t)$ for various $J$ and $t$.}
\end{table}
\end{proposition}

The bounds in Table~\ref{eps J table} are calculated as $|\epsilon_J(t)|\le
c_{\max}(J,t)\mathcal{M}_{\max}(t)$ where $c_{\max}(J,t)$ is essentially 
a bound on the truncation error after $J$ terms 
in the Taylor expansion of $f_s(z)$
at $z=0$,  and $\mathcal{M}_{\max}(t)$ is a
bound on the sum of the $R+1$ blocks.
The bound 
$\mathcal{M}_{\max}(t)$ that we proved is of the form $\ll t^{1/6}\log t$,
which we know is a significant overestimate, 
and is in fact the main source of inefficiency in Proposition~\ref{trunc error
proposition}.
In practice, $\epsilon_J(t)$ is bounded by 
something like 
$c_{\max}(J,t) \sqrt{\log (N/v_0)}$,
though we cannot prove this. In any case, Table \ref{cmax J table} furnishes bounds
on $c_{\max}(J,t)$ alone. The numbers appearing there should be much 
closer to the true truncation error in our computations. 
Note that $c_{\max}(J,t)$ depends mostly on $J$, and 
that its dependence
on $t$ is weak in comparison.

\begin{table}[h]
\small
\label{cmax J table}
\begin{tabular}{|l|l| l| l| l| l|}
\hline
\backslashbox[1cm]{$t$}{$J$} &
 \multicolumn{1}{|c|}{$18$} &
 \multicolumn{1}{c|}{$21$} &
 \multicolumn{1}{c|}{$24$} &
 \multicolumn{1}{c|}{$27$} &
 \multicolumn{1}{c|}{$30$} \\
\hline
$10^{24}$ & $1.048\times 10^{-8}$ & $4.864\times 10^{-10}$ & $1.370\times
10^{-10}$ & $9.399\times 10^{-11}$ & $6.838\times 10^{-11}$\\
$10^{26}$ & $1.028\times 10^{-8}$ & $3.358\times 10^{-10}$ & $2.723\times
10^{-11}$ & $1.397\times 10^{-11}$ & $1.004\times 10^{-11}$\\
$10^{28}$ & $1.025\times 10^{-8}$ & $3.137\times 10^{-10}$ & $1.112\times
10^{-11}$ & $2.223\times 10^{-12}$ & $1.478\times 10^{-12}$\\
$10^{30}$ & $1.024\times 10^{-8}$ & $3.105\times 10^{-10}$ & $8.751\times
10^{-12}$ & $4.989\times 10^{-13}$ & $2.207\times 10^{-13}$\\
$10^{32}$ & $1.024\times 10^{-8}$ & $3.100\times 10^{-10}$ & $8.404\times
10^{-12}$ & $2.459\times 10^{-13}$ & $3.618\times 10^{-14}$\\
$10^{34}$ & $1.024\times 10^{-8}$ & $3.099\times 10^{-10}$ & $8.353\times
10^{-12}$ & $2.087\times 10^{-13}$ & $9.099\times 10^{-15}$\\
$10^{36}$ & $1.024\times 10^{-8}$ & $3.099\times 10^{-10}$ & $8.345\times
10^{-12}$ & $2.033\times 10^{-13}$ & $5.125\times 10^{-15}$\\
\hline
\end{tabular}
\vspace{3mm}
\caption{Bounds on $c_{\max}(J,t)$ for various $J$ and $t$.}
\end{table}

We consider the sensitivity of the main sum to perturbations in the 
quadratic sums in \eqref{quadratic sum formula}. 
Such perturbations arise from 
the accumulation of roundoff errors when using floating point arithmetic.
If each sum $\mathbf{F}(K_r,\mathbf{c}_r;a_r,b_r)$ in \eqref{quadratic sum
formula} is computed to within $\varepsilon$, say, then 
this introduces a total error 
\begin{equation}
\le \varepsilon \sum_{r=0}^R \frac{1}{\sqrt{v_r}}.
\end{equation}
Choosing $u_0$ and $v_0$ as in Proposition~\ref{trunc error proposition},
and noting $v_{r+1} = v_r+K_r\ge v_r(1+1/u_0)$, 
it is easy to show that this error is $\le 0.05 \varepsilon t^{1/6}$, provided 
that $u_0\ge 200$ and $t\ge
10^{10}$ say. In practice though, this maximal size is never reached. 
Instead, due to pseudorandom nature of roundoff errors, one observes
square-root
cancellation in their sum. So the cumulative error is typically 
\begin{equation}
\le \varepsilon \left(\sum_{r=0}^R \frac{1}{v_r}\right)^{1/2}.
\end{equation}
Thus, under the same assumptions on $u_0$ and $t$ as before, 
the error is $\le 0.0011 \varepsilon$.

\section{Computing quadratic exponential sums}\label{section quadratic sums}

The main new ingredient in our algorithm to compute the zeta function is an
implementation of Hiary's algorithm to compute quadratic exponential sums
$\mathbf{F}(K,\mathbf{c};a,b)$. 
The algorithm runs in $O((J+1)^A\log^A(K/\epsilon))$ bit operations, 
where $\epsilon$ is the desired accuracy and $A\le 3$.
This algorithm was derived in \cite{hiary-2} to 
compute $\zeta(1/2+it)$ in $t^{1/3+o(1)}$ time, which is the method
implemented in our computations.
Currently, the fastest method for computing zeta at a single point 
has asymptotic running time 
$t^{4/13+o(1)}$; see \cite{hiary}.  This method relies on computing
cubic exponential sums instead of quadratic sums.

From a high level point of view, in the typical case the algorithm for quadratic
sums works by
using Poisson summation to write $\mathbf{F}(K,\mathbf{c}; a, b)$ as a combination of a shorter
sum and some integrals which can be calculated to whatever precision we like.
The length of the new sum will be $\floor{a + 2bK}$, and we will be able to
assume that $0 \le b < 1/4$, so the length of the sum decreases quickly.  There
are some cases where this Poisson summation does not work well, but these
correspond precisely to the case where $b$ is very small and we can compute
this sum by using Euler-Maclaurin summation.

From \cite{hiary-2}, we distill Proposition~\ref{proposition-theta-sum}.
    The proof of this proposition is essentially the content of Equations (3.37), (3.38), and (3.39)
    \cite{hiary-2}, though we have modified the notation in some ways,
    and have explicitly written out the result for general $\v$, rather than
    only for $\v = (0, 0, \ldots, 1)$, as is done in \cite{hiary-2}.
    (This implicitly involves changing an order of summation, which causes
    the appearance of $z_j$ and $z'_j$ and makes computation
    more efficient.)
The formula in Proposition~\ref{proposition-theta-sum} is 
fairly complicated 
since, following \cite{hiary-2}, it is completely explicit and 
avoids numerical differentiation.
This ensures better and more robust performance 
in practice. 

The algorithm works by applying 
the formula in Proposition~\ref{proposition-theta-sum} 
repeatedly, 
until either $b \ll 1/K$, in which case the Euler--Maclaurin summation 
is used, or $K\le K_{\min}$, in which case direct summation is used. 
Before each application of the formula, 
the linear argument $a$ is normalized to be in $[0,1]$
and the quadratic argument $b$ is normalized to be in $[0,1/4]$.
This is done using
the periodicity of the complex exponential, conjugation, and 
the identity $(k\pm k^2)/2\equiv 0 \pmod 1$, which enables changing $b$
in steps of $1/2$.
The normalization of $b$ is critical to this algorithm as it 
ensures that the length of the quadratic sum is halved
after each iteration; hence, the total 
number of iterations is $\ll \log K$. 

The algorithm for computing quadratic sums has 
been subsequently implemented by Kuznetsov~\cite{kuznetsov}, but using Mordell integral
identities instead of the Poisson summation and relying on 
numerical differentiation.

\begin{proposition}\label{proposition-theta-sum}
If $K$ is a positive integer, 
$a\ge 0$, and $b >0$, then we have the transformation
\begin{align*}
&\mathbf{F}(K,\mathbf{v};a, b) = 
                    \mathbf{F}(q,\v'; a', b') +\mathbf{R}(K,\v;a,b),
\end{align*}
where 
$q = \floor{a + 2bK}$ is the length of the new quadratic sum,
$a' = a/(2b)$ is the new linear argument, 
and $b' = -1/(4b)$ is the new quadratic argument.
The new coefficient vector $\mathbf{v}'=(v'_0,\ldots,v'_J)$ is defined by
\begin{align*}
&    v'_j = e^{-\frac{i \pi a^2}{2b}} q^j b^{-\frac{j}{2}}
                 \sum_{\ell=j}^J v_{\ell} b^{-\frac{1 + \ell}{2}}
                 K^{-\ell}A_{j,\ell}
                 \sum_{\substack{k \equiv \ell - j \bmod 2 \\ 0 \le k \le \ell
                 -j}}
                 a^k b^{-\frac{k}{2}} B_{\ell-j,k},\\
&A_{j,\ell} = \frac{\ell!}{j!} \pi^{ \frac{j-\ell}{2} } 2^{\frac{j - 3\ell -
1}{2}}
e^{\frac{\pi i}{4}(1 + 3(\ell-j))},\qquad
B_{j,k} = \frac{(-1)^{\frac{k + j}{2}}}{\left(\frac{j-k}{2}\right)!\,k!}
(2\pi)^{\frac{k}{2}} e^{-\frac{3\pi i  k}{4}}.
\end{align*}
The remainder $\mathbf{R}(K,\v;a,b)$ 
is given explicitly by
$$
  \mathbf{R}(K,\v;a,b)=          
  e^{2\pi i aK + 2\pi i bK^2}
  \left[e^{\frac{\pi i }{2}}\sum_{j=0}^J S_j
                    + \frac{1}{2}\sum_{j=0}^J v_j\right]
                    +\frac{v_0}{2}
                    - \delta_{\ceil{a}-1}
                    v'_0.$$
The $S_j$ are defined as follows. Let $z_j = i^j \sum_{\ell=j}^J v_\ell \binom{\ell}{j}$, 
$z'_j = \sum_{\ell=j}^J
v_{\ell} \binom{\ell}{j} 2^{\frac{\ell+1}{2}} e^{\frac{i\pi(\ell+1)}{4}}$,
$\omega = \frc{a + 2bK}$, and
$\omega_1 = \ceil{a} - a$. Then
\begin{align*}
    S_j =\,\,
            &z_j
            \Bigg[
                \I_{\tilde C_1}(K, j, \omega, b)
                - \I_{C_7}(K, j, \omega, b)
                - \jbulk(\omega, b, j, q - \ceil{a}, K) \\
                &+ (-1)^j \I_{C_{9H}}(K, j, 1 - \omega, b)
                + \jboundary(2bK - \omega_1, 1 - \omega, b, j, K)
            \Bigg]\\
    &   
            + v_j
            \Bigg[
                (-1)^{j+1}\jbulk(\omega_1, b, j, q - \ceil{a}, K)
                + (-1)^{j+1}\I_{C_7}(K, j, \omega_1, b) \\
                &+ \I_{C_{9H}} (K, j, 1 - \omega_1, b)
                + (-i)^{j+1} \jboundary (2bK - \omega, 1 - \omega_1, b, j, K)
            \Bigg]\\
    &- z'_j e^{2\pi i a K-2\pi \omega K} \I_{C_{9E}} (K, j, \omega, b).
\end{align*}
Here, the $\mathcal{J}$ terms are given by the integrals  
\begin{align*}
&    \jbulk(\alpha, \beta, j, M, K) = \frac{1}{K^j} \int_0^K t^j \exp(-2 \pi
\alpha t - 2 \pi i \beta t^2)
        \frac{1 - \exp(-2 \pi M t)}{\exp(2 \pi t)-1} \ud t,\\
&    \jboundary(\alpha_1, \alpha_2, \beta, j, K) = \frac{1}{K^j} \int_0^K t^j
\exp(-2\pi \beta t^2)
                \frac{\exp(-2\pi \alpha_1 t) + (-1)^{j+1} \exp(-2 \pi \alpha_2 t)}{\exp(2
                \pi t) - 1} \ud t.
\end{align*}
The $\I$ terms are all integrals of the same integrand along different paths in
the complex plane, which we can write explicitly as integrals over segments of
$\mathbb{R}_{\ge 0}$ as
\begin{align*}    
&    \I_{C_7}(K, j, \alpha, \beta) = 
e^{-\frac{i\pi (j+1)}{4}}\frac{1}{K^j}
                        \int_0^{K\sqrt{2}} t^j \exp\left(-(1 +
                        i)\pi\sqrt{2}\alpha t -
                        2\pi \beta t^2\right) \ud t,\\
&\I_{C_{9H}}(K, j, \alpha, \beta) = \frac{1}{K^j} \int_0^\infty t^j \exp(-2 \pi
\alpha t - 2
\pi i \beta t^2) \ud t,\\
&\I_{C_{9E}}(K, j, \alpha, \beta) = \frac{1}{K^j}\int_0^\infty t^j \exp(-2
\pi(\alpha - i\alpha +
2\beta K + 2i\beta Kt - 4 \pi \beta t^2) dt,\\
&\I_{\tilde{C}_1}(K, j, \alpha, \beta) =  -i e^{-2 \pi \alpha K - 2 \pi i \beta
K^2}
        \int_0^K t^j \exp(2 \pi i \alpha t - 4 \pi \beta K t + 2 \pi i \beta t^2) \ud t.
\end{align*}
\end{proposition}
    
Some remarks:
\begin{itemize}
 \item The formula in
 Proposition~\ref{proposition-theta-sum}  takes on a simpler
 form if $J=0$. For example, 
 $$z_0=v_0, \quad z'_0=(1+i)v_0,\quad 
 v'_0=\frac{e^{\frac{\pi i}{4} -\frac{i\pi a^2}{2b}}}{\sqrt{2b}}  v_0.$$
  \item The integral $\J_2$ occurs from certain ``boundary'' terms in the
computation, while the bulk of the contribution to the sum generally comes from
the $\J_1$.
    \item
    As might be expected, and as can be seen in the formula for $S_j$, for
    large $K$ $\mathbf{F}(K,\v; a, b)$ is extremely sensitive to small
    perturbations in $a$ and $b$. 
    This means that we require some sort of
    high precision computation even if we only want a moderate precision
    answer.  A feature of this formula, however, is that it isolates a
    small number of terms which require high precision computation (namely,
    a few exponentials and $a'$ and $b'$).

    \item
    In practice it is better to compute at least some of these integrals
    simultaneously for all $j$. In particular, our methods of computation
    for the $\J$ terms are only weakly dependent on $j$ (e.g. we use Taylor
    expansions to approximate the integrals as polynomials in $t$, and the
    same terms arise many times as integrands). Our current methods to
    compute the $\I$ terms do not have this feature and are not as amenable
    to simultaneous computation, but this is a possible spot for future
    optimization. Another optimization is to precompute 
    the $\J$ and (some of the) $\I$ integrals instead of
    computing them on the fly.

    \item
    If $\omega$ is not too small (compared to the target precision) then the terms
    involving $\I_{C_{9E}}$ and $\I_{\tilde C_1}$ are not large enough to make
    any significant contribution, so we do not need to compute them. This is
    usually the case, and can easily be detected, so on average the computation
    of these terms adds almost nothing to the running time of the algorithm.
\end{itemize}

\section{Simple multi-evaluation}

Thus far we have only described how to evaluate $\zeta(1/2 + it)$ at a
single point. To locate zeros and make plots of zeta, we of
course want to evaluate at more points. We note that the sum
$\mathcal{M}(t)$, which consumes almost all the computation time, 
is a bandlimited function with highest frequency $\tau = \log
\lfloor \sqrt{t/2\pi}\rfloor$;  see \cite{odlyzko-manuscript}.
So  we can use interpolation to
recover $\mathcal{M}(t)$ for any $t$ in a small window if we have already evaluated
it on a relatively coarse grid of points covering a slightly larger window.
Such a grid consists of points spaced $\pi/\beta$ apart, with 
$\beta > \tau$.
In light of this, we focus first on the problem of
computing $\mathcal{M}(t_0 + \delta j)$ for a range of integers $j$, with $t_0$ large and
$\delta j \ll 1$. For concreteness, we might imagine that 
$10^{24} < t_0$,
$\delta = .04$, and $0 \le j \le 1000$, which corresponds to the evaluation
of $\zeta(1/2 + it)$ in a window of size $40$. To further simplify matters, 
we avoid the vicinity of $t_0$ where the length of $\mathcal{M}(t)$
changes; i.e. $t_0$ of the form  $2\pi n^2$ for some integer $n$. 

We perform the multi-evaluation on each block, which we recall have the shape
\begin{align}\label{simple multi eval}
    \sum_{n=v}^{v + K - 1} \frac{e^{i(t_0+\delta j)\log n}}{\sqrt{n}} 
               = \frac{e^{i(t_0+\delta j)\log v}}{\sqrt{v}} 
               \sum_{k=0}^{K - 1} \frac{e^{i(t_0+\delta j)\log (1 +
               k/v)}}{\sqrt{1+k/v}}.
\end{align}
The entirety of our multi-evaluation is based on the simple observation that in
factoring out the first term in this sum we have removed most of the
oscillation from the summands. 
So the inner sum in \eqref{simple multi eval} changes little
with $j$, and it suffices to approximate by 
its value at $t_0$ only.

Let us denote the inner sum in \eqref{simple multi eval} 
by $V(t,v,K)$. 
Then using the inequality $\log(1+x)\le x$ for $0\le x\le 1$,  
and the bound  $(K-1)/v\le 1/u_0$,
we obtain
\begin{equation}
|V(t_0+\delta j,K,v)-V(t_0,K,v)|\le 
\frac{\rho \delta j}{u_0} \max_{0\le \Delta < K} 
\left|\sum_{k=\Delta}^{K-1} e^{it_0\log (v+k)}\right|,
\end{equation}
where $\rho>1$ depends on $t_0$. 
Given the range of $t_0$ under consideration, 
we can show that $\rho=1.1$ is admissible. So, summing over all blocks, 
the total error in the multi-evaluation method is at most
$\rho \delta j u_0^{-1} \mathcal{M}_{\max}(t_0)$.
Now, $\mathcal{M}_{\max}(t_0) \le A t_0^{1/6}\log t_0$ 
for some constant $A$ (see Lemma~\ref{trunc error lemma 2}),
and by the main result in \cite{hiary-corput}
we should be able to take $A \le 1$.
Moreover, we have $u_0\ge \alpha t^{1/3}$ where we took $\alpha = 1/0.9$ in 
our computations. Thus, the error in the multi-evaluation is 
\begin{equation}\label{total multi eval error}
\le \frac{\rho A \delta j\log t_0}{\alpha t_0^{1/6}} \le \frac{\delta j \log
t_0}{t_0^{1/6}}.
\end{equation}
In light of this, the multi-evaluation can be carried out
safely over a large range of $j$ (of length $t_0^{1/6-\epsilon}$). 
In practice, the  estimate \eqref{total multi eval error} is 
 conservative because the estimate for $\mathcal{M}_{\max}(t_0)$ 
 is wasteful. 
Almost surely the cumulative error
will be significantly below the maximal size \eqref{total multi eval error},
and is much closer to  
\begin{equation}
\le \frac{\delta j \sqrt{\log(N/v_0)}}{t_0^{1/3}}.
\end{equation}

\section{Our implementation}\label{section-implentation}

\begin{figure}
    \makebox[\textwidth][c]{\includegraphics[scale=.9]{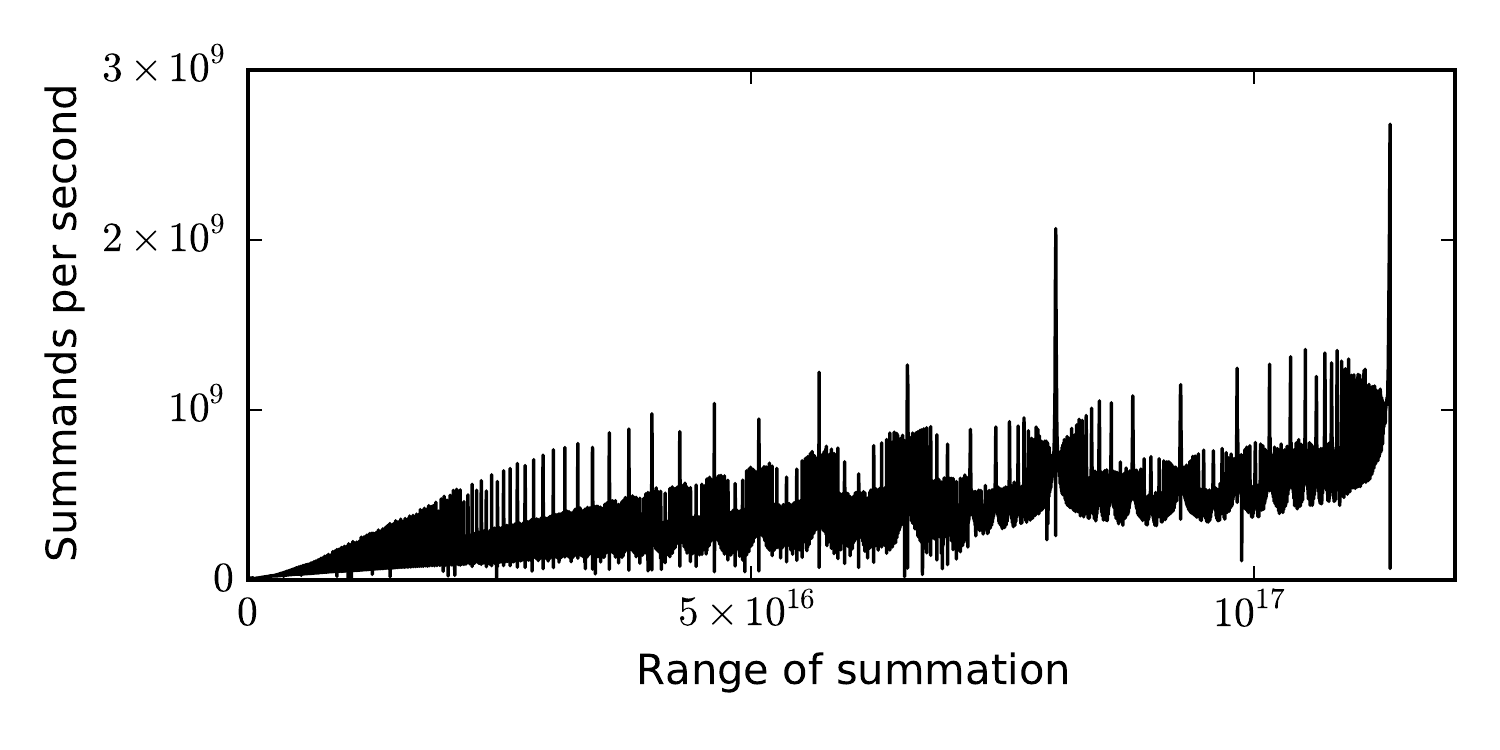}}
    \caption{The computation at the largest height was split into over 130000 independent
        blocks of computation, each computing a consecutive piece of the main sum. Here
        we show the ``speed'' of our implementation in terms of the number of summands
        per second summed in each block. We can see here that the diophantine properties
        of the input affect the running time of the theta algorithm; for example, the
        largest spikes correspond to summands $n^{-1/2 + it}$ where $t/(4\pi n^2)$ is
        very close to an integer.
    }
    \label{figure-computation-speed}
\end{figure}

In this section, we document our implementation of the $t^{1/3+o(1)}$ algorithm.
We remark that due to special 
features of the theta algorithm in \cite{hiary-2} (e.g.\ strong links between the running time
and number-theoretic properties of the inputs), there
are several spots of 
optimization, but we have explored only few of them.

\subsection{Quadratic exponential sums} \label{subsection theta implementation}
We have a general implementation of the algorithm described by
Proposition \ref{proposition-theta-sum}, focused on moderate precision output
for general input. The entirety of our \verb!C++! code currently consists of a bit over 3000 lines,
though some of that is for precomputed tables and constants that are produced
quickly.
It was quite useful during the implementation to constantly compare answers
obtained from the \verb!C++! 
code with answers obtained from a basic version of the
algorithm that was implemented in \verb!Mathematica!.
The end result is a single function which computes the sum $F(K,\v; a,
b)$ to a specified ``attempted'' absolute precision $\epsilon$. By this we
mean that if our implementation were bug-free, and if our subroutines employ
sufficient working precision consistently (as they hopefully should), 
then we would compute the answer to within $\epsilon$.  

In practice, floating point arithmetic is not exact, and we limit our use of
high precision arithmetic (using \verb!MPFR!) 
to small areas of the algorithm where it is completely
unavoidable. For example, if $K$ is very large then $F(K, \v; a, b)$ typically
oscillates rapidly as $b$ changes (on a scale like $1/K^2$), so it is necessary to
specify the value of $b$ sufficiently precisely.
For typical input (for example, with $J = 18$, $\v = (1, 1, \ldots, 1)$,
and $a$ and $b$ arbitrary) we
expect to be able to able to compute up to an absolute precision of around
$10^{-10}$ with our implementation, and we usually get over $40$ bits of relative precision.

There are also some other limitations in our implementation that come from our
use of \verb!C++! doubles. For example, there are numbers that occur in the
computation which may become too large or too small to be represented as double
precision numbers, but in practice our implementation will
usually fail before we reach this point. 
We hope, at least, that our
implementation will return {\tt NaN} in such a case, rather than returning a
wrong answer. 
These are technical limitations, however, and with some effort 
 they could be removed at little cost to the overall running time of our
implementation.

For efficient evaluation of the integrals arising in Proposition
\ref{proposition-theta-sum}, we cannot rely on straightforward numeric
integration. Instead we must deal with each integral on a case-by-case basis
and evaluate as appropriate. Many of the details are given in \cite{hiary-2},
and we do not include them all here, but as a typical example we explain here how one might
evaluate the integral
\[
    \mathcal{J} :=
\jbulk(\alpha, \beta, 0, M, 1) = \int_0^1 \exp(-2 \pi
\alpha t - 2 \pi i \beta t^2)
        \frac{1 - \exp(-2 \pi M t)}{\exp(2 \pi t)-1} \ud t.
\]
Here, $\alpha \in [0,1]$ and $\beta\in [0,1/4]$.
We can begin by replacing $\exp(-2 \pi i \beta t^2)$ by its Taylor series,
and get
\[
    \mathcal{J} = \sum_{r=0}^\infty \frac{(-2\pi i \beta)^r}{r!}
        \int_0^1 t^{2r} \exp(-\pi a t)\frac{1 - \exp(-2 \pi M t)}{\exp(2 \pi t)-1} \ud t.
\]
The integral inside the sum will be small, so we can truncate the infinite sum
after a small number of terms, and we now focus on this inner integral, in which
we can expand the geometric sum
\[
    \frac{1 - \exp(-2\pi M t)}{\exp(2 \pi t) - 1}
\]
to get that the inner integral equals
\[
    \sum_{m=1}^M \int_0^1 t^{2r} \exp(-2 \pi (\alpha + m) t) dt.
\]
This integral is now a fairly simple function, and we have a number
of choices for how to evaluate it. For example, we can evaluate it explicitly
using its antiderivative, or we can again use the Taylor expansion for the
exponential and integrate term-by-term, or we can write it in terms of an
incomplete gamma function and use a continued fraction expansion; these methods
each work well for different ranges of parameters.

This still does not completely solve the problem, as $M$ may be very large. To
deal with this, once $m$ is past a certain size we use a simple approximation
to the antiderivative of the integrand and Euler--Maclaurin summation to
compute the sum over the full range.

\subsection{The main sum}\label{subsection main sum implementation}
We compute the sum $\mathcal{M}(t) = \sum_{n=1}^N n^{-1/2 + it}$ in three
stages, as different methods of computation are appropriate for different sizes
of $n$. We write
$$    \mathcal{M}(t) = \M_1(t) + \M_2(t) + \M_3(t),$$
where
\[
    \M_1(t) = \sum_{n_0 \le n < n_1}\frac{n^{it}}{n^{1/2}},
    \ \ \ \ \M_2(t) = \sum_{n_1 \le n < n_2}\frac{n^{it}}{n^{1/2}},
    \ \ \ \ \M_3(t) = \sum_{n_2 \le n < n_3}\frac{n^{it}}{n^{1/2}},
\]
with $n_0 = 1, n_1 \asymp t^{1/4}, n_2 \asymp t^{1/3}$ and $n_3 = N$. 

In stage
1, we compute the sum $\mathcal{M}_1(t)$ directly. More specifically, we compute each term to
roughly 50 bits of relative precision, and add them up. This ensures that,
for the ranges of $t$ where we computed,
the roundoff error $|\mathcal{M}_1(t)-\fl(\mathcal{M}_1(t))|$
is far subsumed by the ``practical truncation error'' discussed 
in \textsection{\ref{zeta theta link}}. 

In stage 2 we still
add up each term in the sum individually, but gain efficiency by approximating
the exponent $(it-1/2)\log(1+k/v)$ in successive terms using a 
rapidly decaying 
Taylor series instead of computing 
it  directly using expensive multiprecision arithmetic for the
logarithm function. 
Another advantage to this approach is that, now, double precision arithmetic 
 suffices to 
to compute many of the terms in the Taylor expansion, and 
higher precision is only needed for the first few terms (the exact number of these
terms and the needed working precision are both determined by
the attempted absolute precision $\epsilon$ and using formulas 
coded in the implementation).
The overall effect is to ensure that 
the roundoff error from stage 2 
is again subsumed by the practical truncation error.

In stage 3 we approximate by quadratic exponential
sums as described in \textsection{\ref{zeta theta link}} and apply the 
theta algorithm
described in \textsection{\ref{section quadratic sums}}, 
passing 
a requested precision parameter $\varepsilon$ to the algorithm. 
In turn, $\varepsilon$ is passed to
subfunctions in our code in order to determine  
the needed working precision for each subfunction, and so on.
Significant effort was spent on 
sharpening formulas to calculate the working precision 
for each subfunctions. This enabled a numerically more robust 
implementation and facilitated the discovery of 
programming bugs. 
Actually, the precision of 
the theta algorithm in practice is much better than 
the requested precision $\varepsilon$. 
Here is a typical example. On a test
suite of $29952$ sums which might come from a computation of
$\zeta(1/2 + i10^{36})$, when the theta algorithm 
is called with a precision of $\varepsilon = 10^{-5}$, the worst
error is $5.361 \times 10^{-6}$ but the typical error is around
$7.45 \times 10^{-9}$, and the vast majority of the errors are
$< 3 \times 10^{-8}$. 

In summary, we expect the values of $Z(t)$ that we
computed to be accurate to within $\pm 10^{-7}$ typically, and
accurate to within $\pm 10^{-5}$ in the worst cases. The bulk of the error comes
from the practical truncation error discussed in \textsection{\ref{zeta
theta link}}. In comparison, the practical roundoff and multi-evaluation errors
are small.

\subsection{Example running time}
Figure \ref{figure-computation-speed} gives an indication of the speed of our
current implementation. We graph the number of terms per second computed
in large blocks of the main sum at the largest height computation we ran, in a
range around $t = 81029194732694548890047854481676712$. For this computation,
we used approximately 22.5 cpu-core-years on the BlueCrystal Phase 2 cluster at
the University of Bristol, which has 2.8 Ghz Intel Xeon E5462 cpus. Using just
our "stage 2" code, we estimate that the same computation would take around 300
core-years instead. Towards the end of the range the "stage 3" code often has a
speedup over "stage 2" by a factor of around 200.

\section{Proofs}\label{proofs}

\begin{proof}[Proof of Proposition~\ref{trunc error proposition}]
By definition
\begin{equation}
\epsilon_J(t) =  
\sum_{r=0}^R \frac{e^{it\log v_r}}{\sqrt{v_r}}
\sum_{j> J}c_r(j) F(K_r,j;a_r,b_r).
\end{equation}
Applying partial summation to $F(K_r,j;a_r,b_r)$ gives
\begin{equation}
|F(K_r,j;a_r,b_r)| \le \left(\frac{K_r-1}{K_r}\right)^jF_{\max}(K_r;a_r,b_r),
\end{equation}
where
\begin{equation}
F_{\max}(K;a,b)=\max_{0\le \Delta<K} 
\left|\sum_{k=\Delta}^{K-1} e^{2\pi i a k+2\pi i b k^2}\right|.
\end{equation}
In view of this, let us define
\begin{equation}\label{cmax Mmax}
\begin{split}
&c_{\max}(J,t) = \max_{0\le r\le R} \sum_{J< j} \left(\frac{K_r-1}{K_r}\right)^j |c_r(j)|,\\
&\mathcal{M}_{\max}(t) = \sum_{r=0}^R
\frac{F_{\max}(K_r;a_r,b_r)}{\sqrt{v_r}}.
\end{split}
\end{equation}
(The dependence of $c_{\max}(J,t)$ on $t$ is through the coefficients $c_r(j)$.)
By the triangle inequality then,
\begin{equation}\label{eps J}
|\epsilon_J(t)|\le c_{\max}(J,t)\mathcal{M}_{\max}(t).
\end{equation}
Now, recall from \eqref{cj def} that
\begin{equation}\label{trunc error lemma eq 0}
c_r(j) =\left(\frac{K_r}{v_r}\right)^j\frac{f_{s^*}^{(j)}(0)}{j!}.
\end{equation}
By construction $K_r= \lceil v_r/u_0\rceil$, hence,
$(K_r-1)/v_r\le 1/u_0$.
Also, if we write
\begin{equation}
\frac{f_{s^*}(z)}{j!}  = e^{-\frac{s^*z^3}{3}}\frac{e^{\frac{s^*z^3}{3}}f_{s^*}(z)}{j!}
= e^{-\frac{s^*z^3}{3}} \sum_{m=0}^{\infty} d_m z^m,
\end{equation}
then
\begin{equation}\label{trunc error lemma eq 1}
\left|\frac{f^{(j)}_{s^*}(0)}{j!}\right| \le \sum_{\substack{0\le m,h \\
m+3h=j}} \frac{|s|^h}{h!3^h} |d_m|,
\end{equation}
where, by Cauchy's theorem applied with a circle of radius
$\rho> 0$ around the origin,
\begin{equation}\label{trunc error lemma eq 02}
|d_m| \le \frac{1}{2\pi}
\left|\int_{|z|=\rho}\frac{e^{\frac{s^*z^3}{3}} f_{s^*}(z)}{z^{m+1}}dz\right|
\le  \frac{e^{\frac{\rho}{2}+\frac{\rho^2}{4}+|s|\sum_{\ell=4}^{\infty} \frac{\rho^{\ell}}{\ell}}}{\rho^m}.
\end{equation}
Choosing $\rho = 1/|s|^{1/4}$ therefore gives
\begin{equation}\label{trunc error lemma eq 2}
|d_m|\le  |s|^{m/4} e^{\lambda(s)},
\end{equation}
where
\begin{equation}
\begin{split}
\lambda(s) &\le  \frac{1}{2|s|^{1/4}}+\frac{1}{4|s|^{1/2}}+ 
\sum_{4\le \ell} \frac{|s|^{-\ell/4+1}}{\ell}\\
&\le \frac{1}{2|s|^{1/4}}+\frac{1}{4|s|^{1/2}}+
\frac{1}{4(1-|s|^{-1/4})}. 
\end{split}
\end{equation}
Combining \eqref{trunc error lemma eq 2}, \eqref{trunc error lemma eq 1},
\eqref{trunc error lemma eq 0},
and observing that $d_0=1$ and $|d_1|+|d_2|+|d_3|= 49/48$ (since they come 
exclusively from expanding $e^{-\frac{1}{2}(z-z^2/2)}$)
 we obtain (on treating the cases $m=0$, $1\le m\le 3$, and $4\le m$ separately) that 
\begin{equation}
\begin{split}
c_{\max}(J,t) & \le 
 \sum_{\frac{J}{3}< h} \frac{1}{u_0^{3h}}\frac{|s|^h}{h! 3^h}
 + \sum_{J< j}\frac{1}{u_0^j} 
  \sum_{\substack{1\le m \le 3,\, 0\le h\\ m+3h=j}} \frac{|s|^h}{h!3^h}|d_m|
+ e^{\lambda(s)}\sum_{J< j}\frac{1}{u_0^j} 
 \sum_{\substack{4\le m,\, 0\le h\\ m+3h=j}} \frac{|s|^{h+m/4}}{h!3^h}\\
&\le 
\sum_{\frac{J}{3}< h} \frac{1}{h!}\frac{1}{(3\alpha^3)^h} 
+ \frac{49}{48|s|^{1/3}\alpha}\sum_{\frac{J-1}{3}< h }
\frac{1}{h !}\frac{1}{(3\alpha^3)^h}\\
&+\frac{e^{\lambda(s)}}{|s|^{1/3}\alpha^4}\sum_{\frac{J-4}{3}< h }
\frac{1}{h !}\frac{1}{(3\alpha^3)^h}+
e^{\lambda(s)}\sum_{J< j} \frac{1}{\alpha^j}\sum_{\substack{5\le m,\, 0\le
h\\ m+3h=j}} \frac{1}{|s|^{m/12}}\frac{1}{h!}\frac{1}{3^h}\\
&\le
\Big[e^{1/3\alpha^3}-\sum_{0\le h \le \frac{J}{3}} 
\frac{1}{h!} \frac{1}{(3\alpha^3)^h}\Big] 
+\frac{49}{48|s|^{1/3}\alpha}\Big[e^{1/3\alpha^3}-\sum_{0\le
h \le \frac{J-1}{3}} \frac{1}{h!}\frac{1}{(3\alpha^3)^h}\Big]\\ 
&+\frac{e^{\lambda(s)}}{|s|^{1/3}\alpha^4}\Big[e^{1/3\alpha^3}-\sum_{0\le
h \le \frac{J-4}{3}} \frac{1}{h!}\frac{1}{(3\alpha^3)^h}\Big]
+\frac{e^{\lambda(s)+1/3}}{|s|^{5/12}\alpha^{J-1}(\alpha-1)}.
\end{split}
\end{equation}
The proposition now follows on substituting
into \eqref{eps J}
this estimate for $c_{\max}(J,t)$ and the bound for
$\mathcal{M}_{\max}(t)$ from Lemma~\ref{trunc error lemma 2}.
\end{proof}

\begin{lemma}\label{trunc error lemma 2}
Let $\alpha = 1/0.9$, $K_{\min} = 2000$.
If $v_0 = \lceil K_{\min} t^{1/3}\rceil \le \frac{1}{3} \sqrt{t/(2\pi)}$ and $u_0= \alpha
|s|^{1/3}\ge 1$, then
\begin{equation}
\begin{split}
\mathcal{M}_{\max}(t) \le
\sqrt{\frac{1}{u_0}+\frac{1}{v_0}}
&\left(\frac{\log\frac{t}{2\pi v_0^2}}{\log(1+\frac{1}{u_0})}+ 
\frac{t^{1/4}}{3\pi^{1/4}}\frac{\sqrt{2}}{\sqrt{u_0}}
+\frac{2t}{3\pi}\frac{\sqrt{2}}{\sqrt{u_0v_0}v_0} \right.\\
&\left.+\frac{\sqrt{t}}{\sqrt{\pi}}\frac{1}{v_0}
+2u_0+2u_0\log\frac{t}{\pi v_0^2}+5\right).
\end{split}
\end{equation}
\end{lemma}
\begin{proof}
Recall that
\begin{equation}
\mathcal{M}_{\max}(t) = \sum_{r=0}^R
\frac{F_{\max}(K_r;a_r,b_r)}{\sqrt{v_r}}.
\end{equation}
To bound $F_{\max}(K_r;a_r,b_r)$, 
we use the Weyl-van der Corput
Lemma in \cite[Lemma 5]{cheng-graham}.
This lemma gives rise to a certain geometric sum, 
which is in turn bounded 
using the Kusmin--Landau Lemma in \cite[Lemma 2]{cheng-graham}. 
This gives for each positive integer $M$,
\begin{equation}\label{trunc error lemma 2 eq 0}
\begin{split}
F_{\max}(K_r;a_r,b_r)^2 \le &  (K+M)\left( 
\frac{K_r}{M}+\min(1/\pi \|2 b_r\|+1,K_r) \right),
\end{split}
\end{equation}
where $\|x\|$ is the distance to the nearest integer to $x$. 
Choosing $M=K_r$ yields 
\begin{equation}
F_{\max}(K_r;a_r,b_r) \le \sqrt{4K_r + 2K_r\min(1/\pi\|2b_r\|,K_r)}.
\end{equation}
We partition $\{b_r\}_{r=0}^R$ into subsets
$I_{\ell} =\{b_r\,:\, 2|b_r|\in  [\ell-1/2,\ell+1/2), \, 0\le r \le R\}$.
So $\mathcal{M}_{\max} = \sum_{\ell} \mathcal{M}_{\ell}$ where  
\begin{equation}\label{trunc error lemma 2 eq 3}
\mathcal{M}_{\ell}=
\sum_{b_r \in I_{\ell}} \sqrt{K_r/v_r}\sqrt{4+2
 \min(1/\pi \|2b_r\|,K_r)}.
\end{equation}
Let $K_{\ell}^*$ denote the maximum block length $K_r$ associated
with a subset $I_{\ell}$. Analogously define $v_{\ell}^*$ to be the
maximum such $v_r$.
Then appealing to the inequality 
\begin{equation}
K_r/v_r \le \sqrt{1/u_0+1/v_0},
\end{equation}
and the bound 
$\sqrt{x+y}\le \sqrt{x}+\sqrt{y}$ (valid for $x,y\ge 0$), 
we obtain
\begin{equation}\label{trunc error lemma 2 eq 4}
\mathcal{M}_{\ell}\le
\sqrt{1/u_0+1/v_0}\left(2|I_{\ell}|+ \sqrt{2/\pi}
M^*_{\ell}\right)
\end{equation}
where  $|I_{\ell}|$ is the cardinality of $I_{\ell}$ and
\begin{equation}
M_{\ell}^*  = 
\sum_{b_r \in I_{\ell}} \sqrt{\min(1/\|2b_r\|,\pi K^*_{\ell})}.
\end{equation}
We bound $\mathcal{M}_{\ell}^*$ 
in terms of the minimum distance between distinct points 
$b_r,b_{r'}\in I_{\ell}$,
\begin{equation}
\delta_{\ell} = \min_{\substack{b_r,b_{r'} \in I_{\ell}\\b_r\ne b_{r'}}}
|2b_r-2b_{r'}|.
\end{equation}
To this end, proceed
similarly to the proof of \cite[Lemma 3.1]{hybrid-dirichlet} (starting 
with Equation (25) there and using the monotonicity of the $b_r$) to obtain
\begin{equation}\label{trunc error lemma 2 eq 7}
\mathcal{M}_{\ell}^* \le \sum_{0\le w \le 1/2\delta_{\ell}}
\sqrt{2\min\left(1/w\delta_{\ell}, K_{\ell}^*\right)}
 \le \sqrt{2K_{\ell}^*} + 2/\delta_{\ell}. 
\end{equation}
The last inequality follows on 
isolating the term in the sum corresponding to $w=0$, fixing the min 
to be $1/w\delta_{\ell}$ in the remainder of the sum, and 
estimating that by an integral.

At this point, we observe that if 
$\ell < t/2\pi v_R^2-1/2$ or $\ell>t/2\pi v_0^2+1/2$, 
then $I_{\ell}$ is empty. In view of this, 
and since $v_R\le \sqrt{t/2\pi}$, 
 we may restrict the range of summation in 
 \eqref{trunc error lemma 2 eq 3} to $1\le \ell \le \ell^* =t/2\pi
v_0^2+1/2$. 
So, substituting \eqref{trunc error lemma 2 eq 7} into \eqref{trunc error lemma 2 eq 4}, 
 summing over $\ell$, 
and using the obvious formula $\sum_{\ell} |I_{\ell}| = R+1$,
we obtain
\begin{equation}\label{trunc error lemma 2 eq 9}
\mathcal{M}_{\max}(t) \le 
\sqrt{1/u_0+1/v_0}\left(2R+2 + \sum_{1\le \ell\le \ell^*} 
\sqrt{2K_{\ell}^*} +\sum_{1\le\ell\le \ell^*} 2/\delta_{\ell}\right).
\end{equation}
Now, a simple calculation shows that
\begin{equation}
\frac{t}{\pi(2\ell+1)}\le v_r^2 \le \frac{t}{\pi(2\ell -1)},\qquad (b_r\in
I_{\ell}). 
\end{equation}
Hence,
\begin{equation}\label{trunc error lemma 2 eq 10}
v_{\ell}^* \le \frac{\sqrt{t}}{\sqrt{\pi(2\ell-1)}},\qquad\textrm{and so}\qquad
K_{\ell}^*\le \frac{\sqrt{ t}}{\sqrt{\pi(2\ell-1)}}\frac{1}{u_0}+1. 
\end{equation}
So, using the inequality $\sqrt{x+y}\le \sqrt{x}+\sqrt{y}$ once
again, we obtain 
\begin{equation}\label{trunc error lemma 2 eq 11}
\sum_{1\le \ell\le \ell^*} \sqrt{2K_{\ell}^*} \le 
\frac{t^{1/4}}{\pi^{1/4}}\frac{\sqrt{2}}{\sqrt{u_0}}
\left(\frac{1}{3}+ \frac{2}{3} \frac{t^{3/4}}{\pi^{3/4}}\frac{1}{v_0^{3/2}}\right)
+\frac{\sqrt{t}}{\sqrt{\pi} v_0}+1. 
\end{equation}
Here, we additionally estimated $\sum_{1\le \ell \le \ell^*}
1/(2\ell-1)^{1/4} \le 1/3+(2/3)(2\ell^*-1)^{3/4}$, which follows on
 isolating the term corresponding to $\ell=1$ and bounding the rest 
by an integral.
Furthermore, since the sequence $\{b_r\}_{r=0}^R$ is monotonically increasing),
then
\begin{equation}\label{trunc error lemma 2 eq 12}
\delta_{\ell} = \min_{b_r,b_{r+1} \in I_{\ell}} (2b_{r+1}-2b_r)
=\min_{b_r,b_{r+1} \in I_{\ell}} \frac{t}{2\pi}
\frac{K_r(v_{r+1}+v_r)}{v_r^2v_{r+1}^2}.
\end{equation}
Thus, using the inequalities $K_r/v_r \ge 1/u_0$ and 
$v_{r+1} \le v_{\ell}^*$, we arrive at the lower bound
$\delta_{\ell} \ge (2\ell-1)/u_0$. Consequently, 
as $\sum_{1\le \ell\le \ell^*} 1/(2\ell-1) \le
1+ \frac{1}{2}\log(2\ell^*-1)$, we obtain
\begin{equation}\label{trunc error lemma 2 eq 13}
\sum_{1\le \ell \le \ell^*} 2/\delta_{\ell} \le
2u_0+ u_0\log(t/\pi v_0^2).
\end{equation}
Last, a routine application of induction (see \cite[Lemma
3.1]{hiary-simple-alg}) gives
\begin{equation}\label{trunc error lemma 2 eq 14}
R \le \frac{\log(\sqrt{t/2\pi}/v_0)}{\log(1+1/u_0)}+1.
\end{equation}
The claim follows on substituting
\eqref{trunc error lemma 2 eq 14},
\eqref{trunc error lemma 2 eq 13},
and \eqref{trunc error lemma 2 eq 11}
into \eqref{trunc error lemma 2 eq 9}.
\end{proof}

\begin{lemma}\label{theta phase lemma}
If $t > 1$, then
\begin{equation}
\left|\theta(t) -\left(\frac{t}{2} \log \frac{t}{2\pi e}
-\frac{\pi}{8}+\frac{1}{48t}\right)\right|\le 
\left(\frac{49}{640}+\frac{3\zeta(4)}{2\pi^3}\right)\frac{1}{t^3}.
\end{equation}
\end{lemma}
\begin{proof}
By definition, $\theta(t)$ is 
the variation in the argument of $\pi^{-s/2}\Gamma(s/2)$ as 
$s$ varies continuously 
along the line segments from $2$ to $2+it$ to $1/2+it$. Therefore
\begin{equation}\label{theta formula}
\theta(t) =-\frac{t}{2}\log \pi+ \im \log \Gamma(1/4+it/2).
\end{equation}
By Stirling's formula (see \cite[\textsection{4.42}]{titchmarsh2}) 
\begin{equation}
\log \Gamma(s/2) = (s/2-1/2)\log(s/2) -s/2 +\frac{1}{2}\log 2\pi -
\int_0^{+\infty}
\frac{B_1(\{x\})}{x+s/2}dx,
\end{equation}
where $B_1(x)=x-1/2$ is the first Bernoulli polynomial, and $\{x\}$ is the
fractional part of $x$. Now, $s=1/2+it$, and so
\begin{equation}
\im \left((s/2-1/2)\log (s/2)\right) =  
\frac{t}{2}\log|1/4+it/2|-\frac{1}{4}\arg(1/4+it/2).
\end{equation}
And it is routine to show that 
\begin{equation}
\begin{split}
&0\le \log \frac{t}{2} +\frac{1}{8t^2} - \log|1/4+it/2|  
\le \frac{1}{64 t^4},\\
&0\le \left(-\frac{\pi}{2}+\frac{1}{2t}\right) + \arg(1/4+it/2) 
\le \frac{1}{24 t^3}.
\end{split}
\end{equation}
Hence
\begin{equation}
\left|\left(\frac{t}{2}\log \frac{t}{2} + \frac{1}{16t} -\frac{\pi}{8}
+\frac{1}{8t}\right)
-\im \left((s/2-1/2)\log (s/2)\right)\right|\le \frac{7}{384t^3}. 
\end{equation}
Moreover, applying integration by parts twice gives
\begin{equation}
\begin{split}
\int_0^{+\infty}
\frac{B_1(\{x\})}{x+s/2}dx 
&= -\frac{1}{6s}+\frac{1}{2}\int_0^{+\infty}
\frac{B_2(\{x\})}{(x+s/2)^2}dx\\
&=-\frac{1}{6s}+\frac{1}{3}\int_0^{+\infty}
\frac{B_3(\{x\})}{(x+s/2)^3}dx.
\end{split}
\end{equation}
Now, $0\le \frac{1}{6t}+\im \left(\frac{1}{6s}\right) \le
\frac{1}{24t^3}$.
Therefore, 
applying integration by parts once more,
and using the estimates $|B_4(x)|\le \frac{2(4!)\zeta(4)}{(2\pi)^4}$ (see
\cite{rubinstein-computational-methods}) and 
$\int_0^{\infty} dx/|x+s/2|^4 \le \frac{2\pi}{t^3}$ yields 
\begin{equation}
\begin{split}
\left|\frac{1}{3}\int_0^{+\infty}
\frac{B_3(\{x\})}{(x+s/2)^3}dx\right|&\le 
\frac{1}{60|s|^3}+ \left|\frac{1}{4}
\int_0^{+\infty}
\frac{B_4(\{x\})}{(x+s/2)^4}dx\right|\\
&\le \left(\frac{1}{60}+ \frac{3\zeta(4)}{2\pi^3}\right)\frac{1}{t^3}.
\end{split}
\end{equation}
Put together
\begin{equation}
\left|\left( \frac{t}{2}\log \frac{t}{2e} - \frac{\pi}{8}+ \frac{1}{48t}
\right) - \im \log \Gamma(1/4+it/2)\right|
\le \left(\frac{49}{640}+\frac{3\zeta(4)}{2\pi^3}\right)\frac{1}{t^3}.
\end{equation}
The lemma follows on using this in \eqref{theta formula}.
\end{proof}

\bibliographystyle{amsplain}
\bibliography{zetaComp}
\end{document}